\documentclass[a4paper,reqno]{amsart}

\pdfoutput=1

\usepackage{bbm}
\usepackage{mathrsfs}  

\usepackage{lmodern}
\usepackage[dvipsnames,svgnames,x11names,hyperref]{xcolor}
\usepackage{mathtools,url,graphicx,verbatim,amssymb,stmaryrd,nicefrac}
\usepackage[shortlabels]{enumitem}
\setenumerate[1]{label=(\roman*),} 
\setitemize[1]{label=\raisebox{0.25ex}{\tiny$\bullet$}}

\usepackage[pagebackref,colorlinks,citecolor=Mahogany,linkcolor=Mahogany,urlcolor=Mahogany,filecolor=Mahogany]{hyperref}
\usepackage{microtype}
\usepackage[margin=1.26in]{geometry}
\usepackage{tikz, tikz-cd}
\usetikzlibrary{matrix,calc,arrows,patterns,decorations.pathreplacing, decorations.markings}
\usepackage{caption}
\usepackage[capitalise]{cleveref}
\AtBeginDocument{%
   \def\MR#1{}}

\newcommand{\cM}{\mathcal{M}}

\newcommand{\cS}{\mathcal{S}}

\newcommand{\bC}{\mathbb{C}}

\newcommand{\bK}{\mathbb{K}}
\newcommand{\bN}{\mathbb{N}}

\newcommand{\bQ}{\mathbb{Q}}\newcommand{\bR}{\mathbb{R}}

\newcommand{\bZ}{\mathbb{Z}}
\newcommand{\bk}{\mathbbm{k}}

\newcommand{\beq}{\begin{equation*}}
\newcommand{\eeq}{\end{equation*}}      

\newtheorem{theorem}{Theorem}[section]
\newtheorem{lemma}[theorem]{Lemma}
\newtheorem{proposition}[theorem]{Proposition} 
\newtheorem{corollary}[theorem]{Corollary}

\newtheorem*{corollary*}{Corollary}

\newtheorem{atheorem}{Theorem}

\newtheorem{acorollary}[atheorem]{Corollary}

\theoremstyle{definition}

\newtheorem*{definition*}{Definition}
\newtheorem*{convention*}{Convention}

\theoremstyle{remark}
\newtheorem{example}[theorem]{Example} 
 
\newtheorem{remark}[theorem]{Remark}
\newtheorem*{remark*}{Remark}
\newtheorem*{example*}{Example}

\newcommand{\mr}[1]{{\rm #1}}
\newcommand{\ul}[1]{{\underline{#1}}}

\newcommand{\aut}{\mr{aut}}

\newcommand{\ra}{\rightarrow}
\newcommand{\lra}{\longrightarrow}
\newcommand{\xlra}[1]{\overset{#1}{\longrightarrow}}

\newcommand{\longhookrightarrow}{\lhook\joinrel\longrightarrow}

\newcommand{\longtwoheadrightarrow}{\relbar\joinrel\twoheadrightarrow}
\newcommand{\lla}{\longleftarrow}

\newcommand{\fib}{\mr{fib}}

\newcommand{\interior}{\mr{int}}
\newcommand{\Aut}{\mr{Aut}}
\newcommand{\Nrd}{\mr{Nrd}}

\newcommand{\GL}{\mr{GL}}

\newcommand{\Diff}{\mr{Diff}}
\newcommand{\Homeo}{\mr{Homeo}}
\newcommand{\Emb}{\mr{Emb}}

\newcommand{\BlockDiff}{\widetilde{\Diff}}

\newcommand{\hAut}{\mr{hAut}}
\newcommand{\Map}{\mr{Map}}
\newcommand{\BlockBDiff}{B\widetilde{\Diff}}
\newcommand{\BlockBHomeo}{B\widetilde{\Homeo}}
\newcommand{\BDiff}{B\Diff}

\newcommand{\Wh}{\mr{Wh}}

\tikzset{
  symbol/.style={
    draw=none,
    every to/.append style={
      edge node={node [auto=false]{$#1$}}}
  }
}

\newcommand{\circled}[1]{\raisebox{.5pt}{\textcircled{\raisebox{-.9pt} {#1}}}}

\DeclareMathOperator{\hofib}{hofib}

\title[]{Finiteness properties of automorphism spaces of manifolds with finite fundamental group}

\author{Mauricio Bustamante}
\address{Departamento de Matem\'aticas \\ 
	Universidad Cat\'olica de Chile \\
	Santiago \\
	Chile} 
\email{mauricio.bustamante@mat.uc.cl}

\author{Manuel Krannich}
\address{Department of Mathematics \\ Karlsruhe Institute of Technology \\ 76131 Karlsruhe \\ Germany}
\email{krannich@kit.edu}

\author{Alexander Kupers}
\address{Department of Computer and Mathematical Sciences \\ 
	University of Toronto Scarborough \\
	1265 Military Trail, Toronto, ON M1C 1A4 \\
	Canada}
\email{a.kupers@utoronto.ca}

\begin{document}
	
\begin{abstract}
Given a closed smooth manifold $M$ of even dimension $2n\ge6$ with finite fundamental group, we show that the classifying space $\BDiff(M)$ of the diffeomorphism group of $M$ is of finite type and has finitely generated homotopy groups in every degree. We also prove a variant of this result for manifolds with boundary and deduce that the space of smooth embeddings of a compact submanifold $N\subset M$ of arbitrary codimension into $M$ has finitely generated higher homotopy groups based at the inclusion, provided the fundamental group of the complement is finite. As an intermediate result, we show that the group of homotopy classes of simple homotopy self-equivalences of a finite CW complex with finite fundamental group is up to finite kernel commensurable to an arithmetic group.
\end{abstract}

\maketitle 

\tableofcontents

\begin{convention*} During the introduction, $M$ denotes a closed connected smooth manifold of dimension $d\ge6$, unless said otherwise.
\end{convention*}

\medskip

The study of finiteness properties of diffeomorphism groups and their classifying spaces has a long history in geometric topology. Combining work of Borel--Serre \cite{BorelSerre} and Sullivan \cite{Sullivan}, it was known since the 70's that if $M$ is 1-connected, then the group $\pi_0(\Diff(M))$ of isotopy classes of diffeomorphisms is finitely generated and even of finite type, in the following sense:

\begin{definition*} A space $X$ is \emph{of finite type} if it is weakly homotopy equivalent to a CW-complex with finitely many cells in each dimension. A group $G$ is \emph{of finite type} (or \emph{of type $F_\infty$}) if it admits an Eilenberg-MacLane space $K(G,1)$ of finite type.
\end{definition*}

It was also known that this result has its limits: based on work of Hatcher--Wagoner \cite{HatcherWagoner}, Hatcher \cite[Theorem 4.1]{Hatcher} and Hsiang--Sharpe \cite[Theorem 2.5]{HsiangSharpe} showed that this can fail in the presence of infinite fundamental group, for instance for the high-dimensional torus $M=\times^d S^1$. Later Triantafillou \cite[Corollary 5.3]{Triantafillou} weakened the condition on the fundamental group to allow all finite groups (however, there is an issue with the proof; see \cref{sec:finiteness-mcg} for an explanation of the issue and a way to circumvent it).

These results can be interpreted as finiteness results for the fundamental group of the classifying space $\BDiff(M)$ for smooth fibre bundles with fibre $M$. Around the same time as Sullivan's work appeared, Waldhausen \cite{Waldhausen} developed a programme to systematically study these and related classifying spaces. It was known (see \cite[Proposition 7.5]{Hatcher} for a proof outline) that his approach could lead to a proof that also the higher homotopy groups of $B\Diff(M)$ are finitely generated in a certain range of degrees increasing with the dimension of $M$, provided $M$ is $1$-connected, and possibly even when $\pi_1(M)$ is finite (see p.\,16 loc.cit.). This relied on two missing ingredients, which were both provided later: a stability result for pseudoisotopies (proved by Igusa \cite[p.\,6]{Igusa} building on ideas of Hatcher \cite{HatcherSimple}) and a finiteness result for algebraic $K$-theory of spaces (proved by Dwyer \cite[Proposition 1.3]{Dwyer} for $1$-connected $M$ and by Betley \cite[Theorem 1]{Betley} for finite $\pi_1(M)$). The question, however, whether the groups $\pi_k(\BDiff(M))$ are finitely generated beyond this range remained open.

This changed with work of the third-named author \cite{Kupers} who---inspired by work of Weiss \cite{WeissDalian}---combined work of Galatius--Randal-Williams \cite{GRWstable,GRWII,GRWI} with Goodwillie, Klein, and Weiss' embedding calculus \cite{Weiss, GoodwillieKlein} to show that indeed \emph{all} homotopy groups of $\BDiff(M)$ are finitely generated as long as $M$ is $2$-connected. In view of the above mentioned results that hold in a range, one might hope that this $2$-connectivity assumption can be weakened to only requiring the fundamental group $\pi_1(M)$ to be finite. Our main result confirms this in even dimensions.

\begin{atheorem}\label{thm:fg pi}
Let $M$ be a closed smooth manifold of dimension $2n \geq 6$. If $\pi_1(M)$ is finite at all basepoints, then the space $\BDiff(M)$ and all its homotopy groups are of finite type.
\end{atheorem}

This has the following immediate corollary.

\begin{acorollary}For a manifold $M$ as in \cref{thm:fg pi}, the homotopy groups of $\BDiff(M)$ are degreewise finitely generated. The same holds for the homology and cohomology groups with coefficients in any $\bZ[\pi_0(\Diff(M))]$-module $A$ that is finitely generated as an abelian group.
\end{acorollary}

\begin{remark*}\ 
\begin{enumerate}
\item There are variants of \cref{thm:fg pi} for spaces of homeomorphisms and for manifolds $M$ that have boundary or come with tangential structures (see \cref{sec:elaborations}).
\item It follows from \cref{thm:fg pi} that for manifolds $M$ as in the statement, not only is $\BDiff(M)$ of finite type, but also the path-component $\Diff_0(M)\subset \Diff(M)$ of the identity (e.g.\,using \cref{prop:finite-type-fibrations} below). It was known that $\Diff_0(M)$ admits up to weak equivalence a CW-structure with countable many cells (see e.g.\,\cite[Proposition 1.1(1)]{BurgheleaFiniteness}) and that it often cannot have one with finitely many cells \cite[Theorem\,B]{ABK}.
\end{enumerate}
\end{remark*}

\subsection*{Embedding spaces}
\cref{thm:fg pi} also implies a finiteness result for the higher homotopy groups of embedding spaces of manifold triads (\cref{thm:thick-codim-one-or-two}). The following is a special case:

\begin{atheorem}\label{thm:codim-one-or-two} Let $M$ be a compact smooth manifold of dimension $2n\ge6$ and $N\subset \mr{int}(M)$ a compact submanifold. If the fundamental groups of $M$ and $M\setminus N$ are finite at all basepoints, then the groups $\pi_k(\mr{Emb}(N,M),\mr{inc})$ are finitely generated for $k \geq 2$.\end{atheorem}

For submanifolds of codimension at least $3$, the space of embeddings $\Emb(N,M)$ can be studied via the aforementioned embedding calculus. In particular, an experienced user of this calculus will be able to prove \cref{thm:codim-one-or-two} in these codimensions with ease. The novelty of our \cref{thm:codim-one-or-two} is that it applies in \emph{all} codimensions, which allows examples such as the following. 

\begin{example*}For a smooth irreducible projective hypersurfaces $Y \subset \bC P^n$, the fundamental group of the complement is finite \cite[Proposition 4.(1.3), Theorem 4.(1.13)]{Dimca}, so if $n\ge3$ then the groups $\pi_k(\Emb(Y,\bC P^n),\mr{inc})$ are finitely generated for $k \geq 2$. An example of such a $Y$ is given by the Fermat quadric, cut out by the equation $z_0^2 + \cdots + z_n^2 = 0$.
\end{example*}

\subsection*{Mapping class groups}\cref{thm:fg pi} in particular says that for closed smooth manifolds $M$ of dimension $2n\ge6$ whose fundamental group is finite at all basepoints, the group $\pi_1(B\Diff(M))\cong \pi_0(\Diff(M))$ of isotopy classes of diffeomorphisms is of finite type. The proof of this part of the result goes also through for odd-dimensional manifolds, and it involves finiteness results for several variants of the group $\pi_0(\Diff(M))$ (see \cref{sec:mcg-section}). In particular, we show:
\begin{atheorem}\label{thm:mcg}\ 
\begin{enumerate}
\item \label{mainthm-mcg-i} Let $M$ be a closed smooth manifold of dimension $d \geq 6$. If the fundamental group of $M$ is finite at all basepoints, then the group $\pi_0(\Diff(M))$ is of finite type.
\item \label{mainthm-mcg-ii}Let $X$ be a finite CW-complex. If the fundamental group of $X$ is finite at all basepoints, then the groups $\pi_0(\hAut^s(X))$ and $\pi_0(\hAut(X))$ of homotopy classes of (simple) homotopy automorphisms are commensurable up to finite kernel to arithmetic groups.
\end{enumerate}
\end{atheorem}
\begin{remark*}\ 
\begin{enumerate}
\item See \cref{sec:finiteness-properties-groups} for what it means for a group to be commensurable up to finite kernel to an arithmetic group. This property in particular implies that the group is of finite type.
\item The new part of \ref{mainthm-mcg-ii} is the result on groups of \emph{simple} homotopy automorphisms. The part regarding the group $\pi_0(\hAut(X))$ of all homotopy automorphisms was proved by Triantafillou \cite[Theorem 1]{TriantafillouhAut}. Under the additional assumption that $M$ be orientable, item \ref{mainthm-mcg-i} was also stated by Triantafillou \cite[Corollary 5.3]{TriantafillouhAut}, but as explained in \cref{sec:finiteness-mcg} below, the proof had a nontrivial gap. We circumvent this issue in the proof by a new argument that relies on the part of \ref{mainthm-mcg-ii} on simple homotopy automorphisms.
\item There are versions of \cref{thm:mcg} for variants of the groups $\pi_0(\Diff(M))$ and $\pi_0(\hAut^s(X))$ such as the group $\pi_0(\Homeo(M))$ of isotopy classes of homeomorphisms, or the subgroup of $\pi_0(\hAut^s(X))$ that stabilises a set of (co)homology classes (see \cref{sec:mcg-section})
\end{enumerate}
\end{remark*}

\subsection*{On the assumptions}
We conclude with comments on the hypotheses of the main result.

\subsubsection*{Infinite fundamental groups}As mentioned previously, it is known that a finiteness result such as \cref{thm:fg pi} can fail if the fundamental group of $M$ is infinite. We elaborate on some explicit instances of this phenomenon in \cref{sec:infinite-generatedness}.

\subsubsection*{Small dimensions}Baraglia \cite{Baraglia} and Lin \cite{JianfengLin4} gave examples of smooth 1-connected $4$-manifolds $M$ for which $\pi_2(B\Diff(M))$ is not finitely generated, so the analogues of Theorems~\ref{thm:fg pi}  and~\ref{thm:codim-one-or-two} fail in dimension $4$. For $2n=2$, the result is well-known.

\subsubsection*{Odd dimensions}We expect \cref{thm:fg pi} to be also valid in all odd dimensions $d\ge7$. Some steps in our proofs, however, use that $M$ is even-dimensional in an essential way. Most notably, we rely on a general form of Galatius--Randal-Williams' work on stable moduli spaces of even-dimensional manifolds \cite{GRWstable,GRWII,GRWI} and a homological stability result for diffeomorphism groups of even-dimensional manifolds due to Friedrich \cite{Friedrich}. So far, there are no analogues of these results in odd dimensions in a sufficiently general form, although partial results in this direction are available, see \cite{PerlmutterSn,BotvinnikPerlmutter,Perlmutter,HebestreitPerlmutter}.

A further obstacle in potential strategies to extend the proof of \cref{thm:fg pi} to odd dimensions is that \cref{thm:mcg} \ref{mainthm-mcg-i} is not known to hold for manifolds with non-empty boundary, but we expect it to be true. We comment on this point in \cref{rem:relative-arithm}.

\subsection*{Acknowledgements}The authors are grateful to the referees for their comments and suggestions and to Oscar Randal-Williams for his comments, in particular for a suggestion that allowed us to extend \cref{thm:mcg} \ref{mainthm-mcg-i} from even to all dimensions $d\ge6$. MB also thanks Oscar Randal-Williams for enlightening discussions about moduli spaces of manifolds and embedding spaces at the MIT Talbot Workshop in 2019. MB and MK were supported by the European Research Council (ERC) under the European Union's Horizon 2020 research and innovation programme (grant agreement No. 756444). AK acknowledges the support of the Natural Sciences and Engineering Research Council of Canada (NSERC) [funding reference number 512156 and 512250], as well as the Research Competitiveness Fund of the University of Toronto at Scarborough. AK is supported by an Alfred J.~Sloan Research Fellowship.

%\vspace{-0.7cm}

\section{Finiteness properties of groups and spaces}\label{sec:recollections}We begin with preliminaries on finiteness properties of various classes of groups and spaces.

\subsection{Finiteness properties of groups}\label{sec:finiteness-properties-groups}
There are various properties of groups relevant to the proof of \cref{thm:fg pi}. We discuss them in the following subsections.

\subsubsection{Groups of finite type}Recall from the introduction that a group $G$ is \emph{of finite type} if it has an Eilenberg--MacLane space $K(G,1)$ with finitely many cells in each dimension. This in particular implies that $G$ is finitely generated. For abelian groups the converse holds:

\begin{lemma}An abelian group $A$ is of finite type if and only if it is finitely generated.\end{lemma} 
\begin{proof}By the classification of finitely generated abelian groups, it suffices to show that any cyclic group $A$ is of finite type. This is clear if $A$ is infinite since $K(\bZ,1)\simeq S^1$ and follows for instance from \cref{lem:f-infty} below for finite cyclic groups.\end{proof}

\begin{lemma}\label{lem:f-infty} Fix a short exact sequence of groups
	\[1 \lra N \lra G \lra Q \lra 1.\]
If $N$ is of finite type, then $G$ is of finite type if and only if $Q$ is of finite type. Moreover, if $G' \le G$ is a subgroup of finite index, then $G$ is of finite type if and only if $G'$ is.
\end{lemma}
\begin{proof}
Using the fibration sequences $K(N,1)\to K(G,1) \to K(Q,1)$ and $G/G'\to K(G',1)\to K(G,1)$, this follows from \cref{prop:finite-type-fibrations} below.
\end{proof}

\subsubsection{Arithmetic groups}A \emph{linear algebraic group $\mathbf{G}$ over $\bQ$} is an algebraic subgroup of $\mathbf{GL}_n$ defined as the vanishing locus of finitely many polynomial equations with rational coefficients in the entries of the matrices and the inverse of the determinant. We denote by $\mathbf{G}_\bQ \subset \mr{GL}_n(\bQ)$ the discrete group of $\bQ$-points, i.e.~solutions to the above equations in $\mr{GL}_n(\bQ)$, and define $\mathbf{G}_\bZ \coloneqq \mathbf{G}_\bQ \cap \mr{GL}_n(\bZ)$. Following \cite[\S 1.1]{Serre}, we call a group $\Gamma$ \emph{arithmetic} if there exists an embedding $\Gamma \hookrightarrow \mr{GL}_n(\bQ)$ and an algebraic group $\mathbf{G} \subset \mathbf{GL}_n$ over $\bQ$ such that the intersection $\Gamma \cap \mathbf{G}_\bZ$ has finite index in both $\Gamma$ and $\mathbf{G}_\bZ$.

We refer to \cite[\S 1.1--\S 1.3]{Serre} for a list of properties of arithmetic groups. We only need:

\begin{theorem}[Borel--Serre]\label{thm:BorelSerre}Arithmetic groups are of finite type.\end{theorem}
\begin{proof}
Any arithmetic group has a torsion-free subgroup of finite index \cite[1.3 (4)]{Serre}, so by \cref{lem:f-infty} it suffices to show the claim for torsion-free arithmetic groups. This follows from the existence of the Borel--Serre compactification \cite[1.3 (5)]{Serre}.
\end{proof}

\subsubsection{(Virtually) polycyclic and solvable groups}Other classes of groups that will play role in our arguments are polycyclic, polycyclic-by-finite, and solvable groups. To recall their definition, remember that a \emph{subnormal series} of a group $G$ is a sequence of subgroups
\[\{e\} = G_0 \le G_1 \le \cdots \le G_n = G\]
such that $G_{i-1}\le G_i$ is normal for all $i$. The quotients $G_{i}/G_{i-1}$ are the \emph{factors} of the series. In these terms, a group $G$ is called
\begin{enumerate}
		\item \emph{polycyclic} if $G$ admits a subnormal series whose factors are finitely generated abelian,
		\item \emph{polycyclic-by-finite} (or \emph{virtually polycyclic}) if $G$ admits a subnormal series whose factors are finitely generated abelian or finite,
		\item \emph{solvable} if $G$ admits a subnormal series whose factors are abelian.
	\end{enumerate}

\begin{remark}\ 
\begin{enumerate}
\item The above definition of a polycyclic group $G$ as a ``poly-(finitely generated abelian)'' group might look unusual at first sight. Note though that $G$ admits a subnormal series with finitely generated abelian quotients if and only if it admits one with cyclic quotients. We opted for the version of the definition we actually use in our arguments, but nevertheless stick to the more common term ``polycyclic''.
\item Polycyclic-by-finite groups are often defined as groups that admit a polycyclic subgroup of finite index. This definition agrees with ours (see e.g.\,\cite[Corollary 2.7 (a)]{Wehrfritz}).
\end{enumerate}
\end{remark}

To state some of the closure properties of these types of groups, we say that a  class of groups is \emph{generated by} a collection of groups if it is the smallest class of groups that contains the given collection and is closed under taking extensions, quotients, and subgroups.

\begin{lemma}\label{lem:group-classes-closure}
	$\quad$
	\begin{enumerate}
			\item Solvable groups are generated by the class of abelian groups.
		\item Polycyclic groups are generated by the class of finitely generated abelian groups.
		\item Polycyclic-by-finite groups are generated by the class of finitely generated abelian groups and finite groups.

	\end{enumerate}
\end{lemma}

\begin{proof}By definition these classes contain the listed groups, so it suffices to show that solvable and polycyclic(-by-finite) groups are closed under extensions, subgroups, and quotients. The first property is easy to see and the last two follow for instance from \cite[2.1]{Wehrfritz}.
\end{proof}

\subsubsection{Commensurability up to finite kernel}Recall that two groups $G$ and $G'$ are \emph{commensurable up to finite kernel} if they are equivalent with respect to the equivalence relation on the class of groups generated by isomorphism, passing to finite index subgroups, and taking quotients by finite normal subgroups.

\begin{remark}
In \cite{Sullivan} and \cite{Triantafillou}, commensurability up to finite kernel is referred to as just \emph{commensurability} (see \cite{KrannichRandalWilliams} for an elucidation of this). This property is also sometimes called \emph{S-commensurability} (see e.g.~\cite{BauesGrunewald}) or \emph{differing by finite groups} (see e.g.~\cite{Kupers}).\end{remark}
Combining \cref{lem:f-infty} and \cref{thm:BorelSerre}, the following lemma is straightforward. It summarises the properties of groups of finite type that play a role in the body of this work.

\begin{lemma}\label{lem:finite-type-groups-closure}\label{lem:f-infty-commensurable} The class of groups of finite type contains all
\begin{enumerate}
\item polycyclic-by-finite groups and
\item arithmetic groups
\end{enumerate}
and is closed under extensions. Moreover, if $G$ and $G'$ are commensurable up to finite kernel, then $G$ is of finite type if and only $G'$ is of finite type.
\end{lemma}

\subsubsection{Groups that have polycyclic solvable subgroups}\label{def:polyclic-solvable-subgroups}To discuss the final class of groups that features in our later arguments, we say that a group $G$ \emph{has polycyclic solvable subgroups} if every solvable subgroup $H\le G$ is polycyclic. 

\begin{lemma}\label{lem:polycyclic-sovable-subgroups}The class of groups which have polycyclic solvable subgroups contains
\begin{enumerate}
\item polycyclic-by-finite groups and
\item arithmetic groups
\end{enumerate}
and it is closed under extensions and passing to subgroups. Moreover, if $G'\le G$ is a finite index subgroup, then $G$ has polycyclic solvable subgroups if and only if this holds for $G'$. 
\end{lemma}

\begin{proof}We first show the two closure properties. That this class is closed under taking subgroups is clear. To prove the closure property for extensions, take a short exact sequence
	$1 \to N \to G \to G/N \to 1$
	where $N$ and $G/N$ have polycyclic solvable subgroups. For a solvable subgroup $H \subset G$, we have an induced exact sequence
	$1 \to N \cap H \to H \to H/(N \cap H) \to 1$,
	where $N \cap H \subset N$ is solvable because it is a subgroup of $H$, and $H/(N \cap H) \subset G/N$ is solvable because it is a quotient of $H$. Thus both $N \cap H$ and $H/(N \cap H)$ are polycyclic and since polycyclic groups are closed under extensions (see \cref{lem:group-classes-closure}), the same holds for $H$.
	
	To establish the claim regarding finite index subgroups $G'\le G$, it suffices to show that if $G'$ has polycyclic solvable subgroups, then so does $G$. By passing to a finite index subgroup, we may assume that $G'$ is normal in $G$ in which case the claim follows from the closure property for extensions, since $G'$ and $G/G'$ have polycyclic solvable subgroups; the former by assumption and the latter because it is finite.
	
	Since finitely generated abelian and finite groups clearly have polycyclic solvable subgroups, it follows that polycyclic-by-finite groups have polycyclic solvable subgroups. This leaves us with showing that an arithmetic group $\Gamma$ has polycyclic solvable subgroups.  Malcev \cite[Theorem\,2]{Malcev} shows this property for subgroups of $\mr{GL}_n(\bZ)$, and as any arithmetic group is a subgroup of  $\mr{GL}_n(\bZ)$ after taking a finite index subgroup, this implies the general case.	
	\end{proof}

\subsection{Finiteness properties of spaces}\label{sec:finiteness-spaces} Turning from groups to spaces, we first discuss:
\subsubsection{Spaces of finite type}
As in the introduction, we define a space $X$ to be \emph{of finite type} if it is weakly equivalent to a CW-complex with finitely many cells in each dimension.

\medskip

The following appears as Proposition 2.5 and 2.9 in \cite{DrorDwyerKan}.

\begin{proposition}\label{prop:finite-type-fibrations}\label{prop:finite-type-homotopy-groups}
\quad
\begin{enumerate}
\item Let $f\colon X\ra Y$ be a $0$-connected map. If $\hofib_y(f)$ is of finite type for all $y\in Y$, then $X$ is of finite type if and only if $Y$ is of finite type.
\item Suppose $X$ is a connected space such that $\pi_k(X)$ is finitely generated for all $k\ge2$, then $X$ is of finite type if and only if $\pi_1(X)$ is of finite type.
\end{enumerate}
\end{proposition}

\begin{remark}In \cref{prop:finite-type-fibrations} and henceforth, we call a homotopy group $\pi_k(X)$ finitely generated if it is finitely generated \emph{as an abelian group}, not just as a $\bZ[\pi_1(X)]$-module.\end{remark}

We often apply \cref{prop:finite-type-fibrations} to truncations of spaces. Recall that an \emph{$n$-truncation} of a space $X$ is a space $\tau_{\le n}X$ together with a map $\tau_{\le n}\colon X\ra \tau_{\le n}X$ such that at all basepoints the induced map $\pi_k(X)\ra \pi_{k}(\tau_{\le n}X)$ is an isomorphism for $0\le k\le n$ and $\pi_k(\tau_{\le n}X)$ vanishes for $k>n$. Every space has a unique $n$-truncation up to weak homotopy equivalence. %(see e.g.~Chapter VI of \cite{GoerssJardine}).

\begin{corollary}\label{cor:fg-homotopy-implies-fg-homology} Fix $n\ge1$ and let $X$ be a connected space. If $\pi_k(X)$ is of finite type for $ k\le n$, then $\tau_{\le n}X$ is of finite type. Consequently, for every $\bZ[\pi_1(X)]$-module $A$ that is finitely generated as an abelian group, the groups $H_k(X;A)$ are finitely generated for $0 \leq k \leq n$.\end{corollary}
\begin{proof}
To obtain the first part, apply \cref{prop:finite-type-homotopy-groups} (ii) to the truncation $\tau_{\le n}X$. The second part follows because $\tau_{\le n}\colon X\to\tau_{\leq n}X$ induces an isomorphism in homology in degrees $\leq n$, and spaces of finite type have degreewise finitely generated homology groups with coefficients in local systems of the type in the statement, which one sees using cellular homology.
\end{proof}

\subsubsection{From homology groups to homotopy groups}In general, finiteness properties of the homology groups of a space $X$ need not imply similar properties for the homotopy groups. If $\pi_1(X)$ is finite, however, this is often the case, such as in the following lemma.

\begin{lemma}\label{lem:fg-homology-implies-fg-homotopy-finite} Let $X$ be a connected space with $\pi_1(X)$ finite. If $H_k(X;A)$ is finitely generated for all finitely generated $\bZ[\pi_1(X)]$-modules $A$ and $k \leq n$, then the same holds for $\pi_k(X)$.\end{lemma}

\begin{proof}As $\pi_1(X)$ is finite, it suffices to show that for $2\leq k\leq n$ the homotopy groups $\pi_k(X)\cong\pi_k(\widetilde{X})$ of the universal cover $\widetilde{X}$ are finitely generated. As $H_k(X;\bZ[\pi_1(X)]) \cong H_k(\widetilde{X};\bZ)$, the groups $H_k(\widetilde{X};\bZ)$ are finitely generated for $k\leq n$, so $\pi_k(\widetilde{X})$ is finitely generated for $k\leq n$ by the Hurewicz theorem modulo the Serre class of finitely generated abelian groups.
\end{proof}

\cref{lem:fg-homology-implies-fg-homotopy-finite} has the following useful corollary.

\begin{corollary}\label{coro:pi1-finite-pik-fg} If a space $X$ is weakly homotopy equivalent to a connected finite CW-complex with finite fundamental group, then the groups $\pi_k(X)$ are finitely generated for $k \geq 1$.\end{corollary}

\subsubsection{Finiteness properties of section spaces} Fixing a cofibration $A \hookrightarrow B$, a fibration $p \colon E \to B$ with nonempty fibres, and a section $s_0$ of $p|_A \colon p^{-1}(A) \to A$, we denote by $\mr{Sect}_A(p)$ the space in the compact-open topology of sections $s$ of $p \colon E \to B$ such that $s|_A=s_0$.

\begin{lemma}\label{lem:section-spaces-finite-CW}
For a finite CW pair $(B,A)$ of relative dimension $r$ and $p \colon E \to B$ a fibration with fibre $F$, the  following holds:
\begin{enumerate}
\item If at all basepoints $\pi_k(F)$ is finite for all $0\le k \leq r$, then $\pi_0(\mr{Sect}_A(p))$ is finite.
\item If at all basepoints $\pi_k(F)$ is  polycyclic-by-finite for $k=1$ and finitely generated for $2\le k\le r+1$, then at all basepoints $\pi_1(\mr{Sect}_A(p))$ is polycyclic-by-finite.
\item Let $m\geq 2$. If at all basepoints $\pi_k(F)$ is finitely generated  for $m\le k \le m+r$, then at all basepoints the group $\pi_m(\mr{Sect}_A(p))$ is finitely generated.
\end{enumerate}
\end{lemma}
\begin{proof}
First note that if $B=B_1\sqcup B_2$ is a union of connected components, then $\mr{Sect}_A(p)\cong \mr{Sect}_{A\cap B_1}(p|_{B_1})\times \mr{Sect}_{A\cap B_2}(p| _{B_2})$, so we may assume that $B$ is connected. In this case we prove the three assertions simultaneously by induction over the number $n$ of relative cells of $(B,A)$. If $n=0$, then we have $B=A$ and $\mr{Sect}_A(p)=\{s_0\}$.  Assume that the conclusion holds for CW pairs with $(n-1)$ relative cells and suppose that $B = B' \cup D^d$ where $D^d$ denotes a $d$-cell for some $d\leq r$ and that $B'$ is obtained from $A$ by attaching $(n-1)$ cells of dimension $\leq r$. The fibres of the restriction map $\mr{Sect}_A(p) \to \mr{Sect}_A(p|_{B'})$ are either empty or equivalent to the $d$-fold loop space $\Omega^d F$ at some basepoint, so the result follows from the long exact sequence in homotopy groups together with the induction hypothesis and the fact that polycyclic-by-finite groups are by \cref{lem:group-classes-closure} closed under extensions, taking subgroups, and quotients.
\end{proof}

\begin{corollary}\label{cor:section-spaces-F-coconnected}
Let $(B,A)$ be a CW pair of (relative) finite type and $p \colon E \to B$ be a fibration with fibre $F$. Suppose that $\pi_0(F)$ is finite, at all basepoints $\pi_k(F)$ is polycyclic-by-finite for $k=1$ and finitely generated for $k\ge2$, and that $\pi_k(F)$ vanishes for $k>n$. Then at all basepoints $\pi_k(\mr{Sect}_A(p))$ is polycyclic-by-finite for $k=1$ and finitely generated for $k \geq 2$.
\end{corollary}
\begin{proof}
Let $A \subset \mr{sk}_n(B) \subset B$ denote the relative $n$-skeleton. We observe that there is a fibration sequence $\mr{Sect}_{A \cup \mr{sk}_n(B)}(p) \to \mr{Sect}_A(p) \to \mr{Sect}_A(p|_{\mr{sk}_n(B)})$
whose fibre is weakly contractible by obstruction theory and the hypothesis $F\simeq \tau_{\leq n}F$. As $(\mr{sk}_n(B),A)$ is a finite CW pair since $(B,A)$ is of relative finite type, the result follows from \cref{lem:section-spaces-finite-CW}.
\end{proof}

\section{Finiteness properties of mapping class groups}\label{sec:finiteness-mcg-emb}
This section serves to establish finiteness results for variants of mapping class groups.

\subsection{Finiteness properties of homotopy mapping class groups}\label{sec:mcg-section}Building on work of Sullivan \cite[Theorem 10.3]{Sullivan}, Triantafillou \cite[Theorem 1]{TriantafillouhAut} proved that the homotopy mapping class group $\pi_0(\hAut(X))$ of a connected finite CW-complex $X$ with finite fundamental group is commensurable up to finite kernel with an arithmetic group. Here $\hAut(X) \subset \Map(X,X)$ is the group-like topological monoid of homotopy self-equivalences of $X$ equipped with the compact-open topology. The proof of \cref{thm:fg pi} relies on two extensions of her result.

\subsubsection{Stabilisers of (co)homology classes} The first extension is minor. It concerns concerning stabilisers of sets of twisted (co)homology classes. To state it, let $X$ be a connected based space and write $\hAut_*^{\pi_1}(X)\subset \hAut(X)$ for the group-like submonoid of pointed homotopy automorphisms that induce the identity on $\pi_1(X)$. This monoid naturally acts on the (co)homology groups $H_k(X;V)$ and $H^k(X;V)$ with coefficients in any $\bQ[\pi_1(X)]$-module $V$. For sets of (co)homology classes $h_*\subset H_*(X;V)$ and $h^*\subset H^*(X;V)$, we write 
\[\hAut^{\pi_1}_*(X)_{h}\subset \hAut^{\pi_1}_*(X)\]
for the pointwise simultaneous stabiliser of $h_*$ and $h^*$.

\begin{proposition}\label{prop:triantafillou-haut}Let $X$ be a connected based space, $V$ be a $\bQ[\pi_1(X)]$-module, and $h_*\subset H_*(X;V)$ and $h^*\subset H^*(X;V)$  be subsets of (co)homology classes. If
\begin{enumerate}
\item $X$ is homotopy equivalent to a CW-complex that has either finitely many cells or finitely generated homotopy groups that vanish aside from finitely many degrees,
\item the group $\pi_1(X)$ and the sets $h_*$ and $h^*$ are finite, and
\item $V$ is finite-dimensional as a $\bQ$-vector space,
\end{enumerate}
then the group $\pi_0(\hAut^{\pi_1}_*(X)_h)$ is commensurable up to finite kernel with an arithmetic group.
\end{proposition}

\begin{proof}We show this by adapting Triantafillou's argument in \cite{TriantafillouhAut}. First we assume $h_*,h^*\subset \{0\}$, in which case $\pi_0(\hAut^{\pi_1}_*(X)_h)=\pi_0(\hAut^{\pi_1}_*(X))$. Abbreviating $\pi_1\coloneqq \pi_1(X)$, we follow \cite{TriantafillouhAut} and consider the zig-zag
\begin{equation}\label{equ:zig-zag-Triantafillou}\aut^{\pi_1}(\cM)\cong \pi_0(\hAut^{B\pi_1}_*(X^{\mr{fib}}_\bQ))\lla\pi_0(\hAut_*^{B\pi_1}(X))\xlra{\cong}\pi_0(\hAut^{\pi_1}_*(X)).
\end{equation}
Here we use the following notation:
\begin{itemize}
\item $\hAut_*^{B\pi_1}(X)$ is the space of pointed homotopy automorphisms of $X$ that commute with a model of the $1$-truncation $X\ra B\pi_1(X)$ as a fibration. This maps to $\hAut^{\pi_1}_*(X)$ by the evident forgetful map. The latter is an equivalence since its fibres are equivalent to loop spaces (at various basepoints) of the mapping space $\Map_\ast(X,B\pi_1(X))$ which is homotopy discrete by obstruction theory.
\item $X_\bQ^{\mr{fib}}$ is the fibrewise rationalisation of $X\ra B\pi_1(X)$  (see p.\,3397 loc.cit.).
\item $\hAut_*^{B\pi_1}(X^{\mr{fib}}_\bQ)$ is defined analogously to $\hAut_*^{B\pi_1}(X)$. The left map is induced by fibrewise rationalisation.
\item $\cM$ is a $\pi_1(X)$-equivariant minimal cdga-model of the universal cover $\widetilde{X}$ together with a $\pi_1(X)$-equivariant weak equivalence $\rho\colon\cM\ra A_{PL}(\widetilde{X})$ to the cdga of PL-polynomial differential forms on $\widetilde{X}$ with its $\pi_1(X)$-action by functoriality  (see p.\,3393 loc.cit.)
\item $\aut^{\pi_1}(\cM)$ are the equivariant homotopy classes of equivariant cdga automorphisms of $\cM$. This group is isomorphic to $\pi_0(\hAut^{B\pi_1}_*(X^{\mr{fib}}_\bQ))$ (see p.\,3393 and 3397 loc.cit.). 
\end{itemize}
By Theorems 6 (i) and 11 loc.cit., $\aut^{\pi_1}(\cM)$ is a linear algebraic group over $\bQ$, the image 
\[\Gamma\coloneqq\mr{im}\big(\pi_0(\hAut_*^{B\pi_1}(X))\ra \aut^{\pi_1}(\cM)\big)\] 
is an arithmetic subgroup, and the map $\pi_0(\hAut_*^{B\pi_1}(X))\ra \pi_0(\hAut^{B\pi_1}_*(X^{\mr{fib}}_\bQ))$ has finite kernel, so we conclude that $\pi_0(\hAut_*^{B\pi_1}(X))\cong \pi_0(\hAut^{\pi_1}_*(X))$ is commensurable up to finite kernel to an arithmetic group which finishes the proof in the case $h_*,h^*\subset\{0\}$.

In the next step, we prove the case where $h^*$ is finite and $h_*\subset \{0\}$. Note that all groups in the sequence act compatibly on the cohomology groups
\[H^k(\cM\otimes_{\bQ[\pi_1]}V)\cong H^k(A_{PL}(\widetilde{X})\otimes_{\bQ[\pi_1]}V)\cong H^k(X;V),\]
so the subgroup $\pi_0(\hAut^{\pi_1}_*(X)_h)$ is commensurable up to finite kernel with the intersection of $\Gamma$ with the subgroup $\aut^{\pi_1}(\cM)_{h}\subset \aut^{\pi_1}(\cM)$ of those automorphism that fix $h^*$. By a straightforward extension of the proof of Theorem 6 (ii) loc.cit., the action map $\aut^{\pi_1}(\cM)\ra \mr{GL}(H^k(\cM\otimes_{\bQ[\pi_1]}A))$ is a map of algebraic groups, so the stabiliser of any cohomology class is an algebraic subgroup. As algebraic subgroups are closed under finite intersections, it follows that $\aut^{\pi_1}(\cM)_{h}\subset \aut^{\pi_1}(\cM)$ is an algebraic subgroup. Since the intersection of an arithmetic subgroup with an algebraic subgroup is arithmetic (see e.g.\,\cite[p.\,106]{Serre}), this implies that $\Gamma\cap \aut^{\pi_1}(\cM)_{h}$ is arithmetic and thus finishes the proof in the case $h_*\subset\{0\}$. 

The general case follows by using the canonical isomorphism $H_k(M;A)\cong H^k(M;A^\vee)^\vee$ where $(-)^\vee$ denotes taking $\bQ$-duals, combined with the fact that algebraic subgroups are closed under finite intersections.
\end{proof}

\subsubsection{Simple homotopy automorphisms} 
Our second extension of Triantafillou's work is more substantial and concerns the subgroup $\pi_0(\hAut^s(X))\le \pi_0(\hAut(X))$ of simple homotopy automorphisms of a finite CW-complex $X$. The following gives the part of \cref{thm:mcg} \ref{mainthm-mcg-ii} about \emph{simple} homotopy automorphisms; the part about \emph{possibly non-simple} ones is \cite[Theorem 1]{TriantafillouhAut}.

\begin{theorem}\label{thm:simple-htpy-mcg}
For a finite CW-complex $X$ with finite fundamental group at all basepoints, the group $\pi_0(\hAut^s(X))$ is commensurable up to finite kernel to an arithmetic group. Moreover, assuming in addition that $X$ is connected and based, the statement of \cref{prop:triantafillou-haut} also holds for the subgroup $\pi_0(\hAut^{\pi_1,s}_*(X)_h)\le \pi_0(\hAut^{\pi_1}_*(X)_h)$ of simple homotopy automorphisms.
%Moreover, if $X$ is a smooth manifold, the same holds for the stabiliser \[\pi_0(\hAut^s(X))_{[T^sX]}\le \pi_0(\hAut^s(X))\] of the stable tangent bundle $[T^sX] \in[X,BO]$.
\end{theorem}

\begin{proof}
Recording the effect of a homotopy equivalence on components gives a map from $\pi_0(\hAut^s(X))$ to the symmetric group on the set $\pi_0(X)$, so its kernel $\sqcap_{[x]\in\pi_0(X)}\pi_0(\hAut^s(X_x))\le \pi_0(\hAut^s(X)$ has finite index as $\pi_0(X)$ is finite; here $X_x$ denotes the component of $x\in X$. The property of being commensurable up to finite kernel to an arithmetic group is closed under finite products, so to show the claim we may assume that $X$ is connected.

In the connected case, we fix a basepoint $\ast\in X$  and note that since the group $\pi \coloneqq \pi_1(X)$ is finite and thus has finitely many automorphisms, we may replace the group $\pi_0(\hAut^s(X))$ up to commensurability up to finite kernel by the group $\pi_0(\hAut^{s,B\pi}_\ast(X))$ of pointed simple self-equivalences over $B\pi$. The latter is the kernel of the homomorphism $\tau\colon \pi_0(\hAut^{B\pi}_\ast(X))\ra \Wh(\pi)=K_1(\bZ[\pi])/\langle \pm\pi \rangle$ that takes a self-equivalence to its Whitehead torsion; the fact that this is a homomorphism follows from the composition formula for Whitehead torsion \cite[Lemma 7.8]{Milnor}. To show that $\ker(\tau)$ is commensurable up to finite kernel to an arithmetic group, we establish a commutative diagram 
\begin{equation}
\label{equ:rational-wh-torsion-square}\begin{tikzcd}
 \pi_0(\hAut^{B\pi}_\ast(X))\arrow[d,"(-)^{\fib}_\bQ",swap] \arrow[rr,"\tau"]&& \Wh(\pi)\dar\\
 \pi_0(\hAut^{B\pi}_\ast(X_\bQ^\fib))\rar{\overline{\tau}}&K_1(\bQ[\pi])\arrow[twoheadrightarrow,r]&\Wh_\bQ(\pi).
\end{tikzcd}
\end{equation}
where $\Wh_\bQ(\pi)\coloneqq K_1(\bQ[\pi])/\langle \pm\pi \rangle$. Here the right-hand vertical map is induced by the  inclusion $\bZ\subset \bQ$, the two-headed map is the quotient map, the left vertical map $\smash{(-)^{\fib}_\bQ}$ is induced by fibrewise rationalisation over $B\pi$, and the map $\overline{\tau}$ will be explained later.

Both the bottom right map and the right-hand vertical map have finite kernel: the former since $\pi$ is finite and the latter because $K_1(\bZ[\pi])\ra K_1(\bQ[\pi])$ has finite kernel \cite[p.\,550]{Bass}. The group $\pi_0(\hAut^{s,B\pi}_\ast(X))$ is thus commensurable up to finite kernel to the kernel of the composition $ \pi_0(\hAut^{B\pi}_\ast(X))\ra K_1(\bQ[\pi])$, so we may show the claim for this kernel instead. Recall from the proof of \cref{prop:triantafillou-haut} that $\pi_0(\hAut^{B\pi}_\ast(X_\bQ^\fib))$ is a linear $\bQ$-algebraic group which contains the image of $\smash{(-)^{\fib}_\bQ}$ as an arithmetic subgroup. We will show below that the kernel of $\overline{\tau}$ is a linear $\bQ$-algebraic subgroup of $\pi_0(\hAut^{B\pi}_\ast(X_\bQ^\fib))$. The intersection of an arithmetic subgroup with a linear algebraic subgroup is arithmetic (see e.g.\,\cite[p.\,106]{Serre}), so this will imply that the image of the kernel of $\pi_0(\hAut^{B\pi}_\ast(X))\ra K_1(\bQ[\pi])$ in $\pi_0(\hAut^{B\pi}_\ast(X_\bQ^\fib))$ is an arithmetic group, which will in turn give the claim on $\pi_0(\hAut^{s,B\pi}_\ast(X))$ since the map $\smash{(-)^{\fib}_\bQ}$ has finite kernel (see the proof of \cref{prop:triantafillou-haut}). We are thus left with defining $\overline{\tau}$, showing that it makes \eqref{equ:rational-wh-torsion-square} commute, and proving that the kernel of $\overline{\tau}$ is a linear $\bQ$-algebraic subgroup.

Writing $d \coloneqq \dim(X)$, the map $\overline{\tau}$ is defined as a composition
\begin{equation}\label{equ:tau-rational}\textstyle{\pi_0(\hAut^{B\pi}_\ast(X_\bQ^\fib))\lra \bigsqcap_{i=0}^{d}\Aut_{\bQ[\pi]}(H_i(\widetilde{X};\bQ))}\xrightarrow{\bigsqcap_i[-]} K_1(\bQ[\pi])^{d+1}\xlra{\chi}K_1(\bQ[\pi]).\end{equation}
Here the first map is given by the action on the homology groups $\smash{H_\ast(\widetilde{X}^\fib_\bQ;\bQ)\cong H_\ast(\widetilde{X};\bQ)}$ of the universal cover. The second map is on each factor an instance of the canonical map $[-]\colon \Aut_R(P)\ra K_1(R)$ for a ring $R$ and a finitely generated projective $R$-module $P$ (see e.g.\,\cite[p\,26]{MilnorKtheory}; note that any $\bQ[\pi]$-module is projective as $\bQ[\pi]$ is semisimple), and $\chi$ sends $(x_0,\ldots, x_d)$ to $\sum_{i=0}^d(-1)^i x_i$. The fact that $\overline{\tau}$ makes the diagram \eqref{equ:rational-wh-torsion-square} commute is the content of \cref{lem:rational-wh-torsion} below, so we are left with showing that $\ker(\overline{\tau})\le \pi_0(\hAut^{B\pi}_\ast(X_\bQ^\fib))$ is a $\bQ$-algebraic subgroup. To do so, we postcompose $\overline{\tau}$ with the \emph{reduced norm} map to the units of the centre
\begin{equation}\label{equ:Nrd-k-grp}
	\Nrd_{\bQ[\pi]}\colon K_1(\bQ[\pi])\longhookrightarrow Z(\bQ[\pi])^\times,
\end{equation}
which is a monomorphism \cite[p.\,594]{WallNorms}. It is induced by morphisms of the form
\begin{equation}\label{equ:Nrd-groupring}
	\Nrd_{\bQ[\pi]}\colon \GL_n(\bQ[\pi])\lra Z(\bQ[\pi])^\times
\end{equation}
for $n\ge0$ that are compatible with stabilisation. Being the units in a finite-dimensional $\bQ$-algebra,  $Z(\bQ[\pi])^\times$ is a linear $\bQ$-algebraic group, so it suffices to prove that the composition of \eqref{equ:tau-rational} with \eqref{equ:Nrd-k-grp} is $\bQ$-algebraic, i.e.\,a morphism of algebraic groups defined over $\bQ$. The first map in the composition \eqref{equ:tau-rational} is $\bQ$-algebraic (see the proof of \cref{prop:triantafillou-haut}). As $\chi\colon (Z(\bQ[\pi])^\times)^{d+1}\ra Z(\bQ[\pi])^\times$ is $\bQ$-algebraic as a composition of multiplications and taking inverses in the commutative $\bQ$-algebraic group $Z(\bQ[\pi])^\times$, it suffices to show that the maps \eqref{equ:Nrd-groupring} are $\bQ$-algebraic, since we may factor the second map in \eqref{equ:tau-rational} as a product, over the homological degree, of the precomposition of the $\bQ$-algebraic map $\smash{\Aut_{\bQ[\pi]}(H_i(\widetilde{X};\bQ))\ra \GL_{n_i}(\bQ[\pi])}$ induced by writing $H_i(\widetilde{X};\bQ)$ as a summand of a finitely generated free module of some rank $n_i$, with the reduced norm \eqref{equ:Nrd-groupring}.

We can thus finish the proof by recalling a construction of the maps \eqref{equ:Nrd-groupring} that makes it clear that they are $\bQ$-algebraic. In fact, more naturally these maps are defined over a finite extension of $\bQ$, but this suffices, by restricting scalars (cf.~\cite[p.\,106]{Serre}). By Maschke's theorem, the $\bQ$-algebra $\bQ[\pi]$ is semisimple and thus by the Artin--Wedddenburn theorem there is an isomorphism
\[\textstyle{\bQ[\pi]\cong \bigsqcap_{i=1}^rM_{m_i}(D_i)}\]
 to a product of matrix algebras over division algebras $D_i$. Writing $\bk_i\coloneqq Z(D_i)$ for the centre (a finite field extension of $\bQ$), the norm map \eqref{equ:Nrd-groupring} is defined as a product
\[\textstyle{\GL_n(\bQ[\pi])\cong \bigsqcap_{i=1}^r\GL_{n\cdot m_i}(D_i) \xrightarrow{\sqcap_{i=0}^r\Nrd_{D_i}} \bigsqcap_{i=1}^r\bk_i^\times \cong Z(\bQ[\pi])^\times}\]
of reduced norm maps 
\begin{equation}\label{equ:Nrd-division}\Nrd_{D_i}\colon \GL_{n\cdot m_i}(D_i)\ra \bk_i^\times.\end{equation}
The latter will be defined in such a way so that it is clear that they are algebraic over $\bk_i$, and this implies the same for \eqref{equ:Nrd-groupring} by choosing a common finite extension of the $\bk_i$'s. To define $\Nrd_{D_i}$, one chooses a splitting field $\bK_i$ for $D_i$, i.e.\,a field extension $\bk_i\subset \bK_i$ with an isomorphism $\phi_i\colon \bK_i\otimes_{\bk_i}D_i\cong M_{l_i}(\bK_i)$ as $\bk_i$-algebras for some $l_i\ge0$ (see e.g.\,\cite[Theorem 2.2.1]{GilleSzamuely}). One then considers for $m\ge0$ the composition
\[M_m(D_i)\xrightarrow{\phi_i(1\otimes(-))}M_{m\cdot l_i}(\bK_i)\xlra{\det}\bK_i\]
which is multiplicative and turns out to, firstly, have image in $\bk_i\le \bK_i$, secondly, be independent of the choice of $\bK_i$ and $\phi_i$, and, thirdly, be given as a homogenous polynomial over $\bk_i$ in a $\bk_i$-basis of $M_m(D_i)$ induced by a $\bk_i$-basis of $D_i$ (see e.g.\,\cite[Construction 2.6.1]{GilleSzamuely} and \cite[p.\,27]{PlatonovRapinchuk}). Choosing $m=n\cdot m_i$ and taking units, this gives \eqref{equ:Nrd-division}, which is an algebraic map over $\bk_i$ as a result of the third property just discussed. This completes the proof of the first part.

Regarding the addendum, recall from the proof of \cref{prop:triantafillou-haut} that the group $\pi_0(\hAut^{\pi}_*(X)_h)$ is commensurable up to finite kernel to the arithmetic group given as the image of $\pi_0(\hAut^{B\pi}_*(X)_h)$ in the $\bQ$-algebraic group $\smash{\pi_0(\hAut^{B\pi}_*(X^{\fib}_\bQ)_h)}$, so, by the first part of this proof, the intersection of this group with the $\bQ$-algebraic subgroup $\ker(\overline{\tau})\le \smash{\pi_0(\hAut^{B\pi}_*(X^{\fib}_\bQ))}$ is commensurable up to finite kernel to $\pi_0(\hAut^{\pi,s}_*(X)_h)$. As intersections of arithmetic subgroups with $\bQ$-algebraic subgroups are arithmetic (see e.g.\,\cite[p.\,106]{Serre}), this gives the claim.
\end{proof}

\begin{lemma}\label{lem:rational-wh-torsion}The diagram \eqref{equ:rational-wh-torsion-square} is commutative.
\end{lemma}

\begin{proof}We begin with a general discussion on torsion invariants over an associative unital ring $R$ which we assume to have the invariant basis number property, following Milnor \cite[p.\,362]{Milnor}. We require all modules be left modules and all chain complexes to be graded by the integers, bounded, and degreewise finitely generated. Given a quasi-isomorphism $\varphi\colon C_\ast \ra D_\ast$ of based chain complexes over $R$ (i.e.\,$C_i$ and $D_i$ are free with specified basis), its mapping cone is an acyclic based chain complex, so it has a \emph{torsion} invariant \[\tau_R(\varphi) \in \widetilde{K}_1(R)=K_1(R)/\langle \pm 1 \rangle\] in the reduced first $K$-group of $R$ (see e.g.~Section 3 loc.cit.). By an observation of Gersten \cite{Gersten}, if source and target agree (i.e.\,$C_\ast= D_\ast$) then $\tau_R(\varphi)$ does, firstly, not dependent on the chosen bases, and can, secondly, be extended to self-equivalences of chain complexes that are only degreewise projective, not necessarily free. (In fact $\tau_R(-)$ can even be extended to have values in the non-reduced group $K_1(R)$ but we will not need this and consider $\tau_R(-)$ as valued in $\widetilde{K}_1(R)$.)
%(Note that the sign difference in Gersten's formula for $\varphi(-)$ \cite[412]{Gersten} compared to Milnor's  \cite[365]{Milnor} is due to their different way of ordering the basis of a mapping cone)
Viewed as an invariant of quasi-automorphisms $\varphi\colon C_\ast\ra C_\ast$ of degreewise projective chain complexes, the torsion $\tau_R(-)$ enjoys the following properties:
\begin{enumerate}
\item\label{prop-torsion-i} We have $\tau_R(\varphi)=0$ whenever $C_\ast$ is acyclic. This follows from \cite[Proposition 1]{Gersten}.
\item\label{prop-torsion-ii}  For a map of short exact sequences
\[
\begin{tikzcd}
0\rar&C_\ast\arrow[d,"\varphi_C"',"\simeq"]\rar{\iota}&D_\ast\arrow[d,"\varphi_D"',"\simeq"]\rar{\pi}&E_\ast\arrow[d,"\varphi_E"',"\simeq"]\rar&0\\
0\rar&C_\ast\rar{\iota}&D_\ast\rar{\pi}&E_\ast\rar&0
\end{tikzcd}
\]
we have $\tau_R(\varphi_D)=\tau_R(\varphi_C)+\tau_R(\varphi_E)$. This follows from an application of \cite[Theorem 3.1]{Milnor}, after making the complexes degreewise free by adding complements, choosing bases, and taking mapping cones.
\item\label{prop-torsion-iii}  We have $\varphi_R(s^n\varphi)=(-1)^n\varphi_R(\varphi)$ for $n\in \bZ$ where $s^n(-)$ shifts chain complexes up by $n$ degrees. This is clear from the definition \cite[p\,412]{Gersten}, but also follows from \ref{prop-torsion-i} and \ref{prop-torsion-ii}.
\item\label{prop-torsion-iv}  Specialising $\tau_R(-)$ to automorphisms $\varphi\colon P\ra P$ of finitely generated projective modules viewed as complexes concentrated in degree $0$, we have $\smash{\tau_R(\varphi)=[\varphi]\in \widetilde{K}_1(R)}$ where $[-]\colon \Aut_R(P)\ra K_1(R)$ is the canonical map which featured in the previous proof. 
\end{enumerate}
If the ring $R$ is semisimple (i.e.~if all its modules are projective), then the invariant $\varphi_R(-)$ simplifies significantly: for any quasi-isomorphism $\varphi\colon C_\ast\ra C_\ast$ one has \[\textstyle{\tau_R(\varphi)=\sum_{i\in\bZ}(-1)^i\big[H_i(\varphi)\colon H_i(C_\ast)\xrightarrow{\cong}H_i(C_\ast)\big]}\]
where $H_i(C_\ast)$ is considered as concentrated in degree $0$. This follows from an induction over the number of nontrivial homology groups using \ref{prop-torsion-i}--\ref{prop-torsion-iv}: Property \ref{prop-torsion-i} gives the initial case. For the induction step, we may by \ref{prop-torsion-iii} assume that $H_0(C_*)$ is the nontrivial homology group of lowest degree. By considering the subcomplex $C'_* \le C_*$ that agrees with $C_*$ for $*>0$, with $\mr{im}(d\colon C_1\ra C_0)$ for $*=0$, and is $0$ otherwise, we get a short exact sequence $0 \to C'_* \to C_* \to D_\ast \to 0$ with $D_\ast\coloneqq C_*/C'_*$, and a self-map of this exact sequence consisting of quasi-isomorphisms $\varphi_{C'}$, $\varphi_C=\varphi$, and $\varphi_D$, so $\tau_R(\varphi_C) = \tau_R(\varphi_{C'}) + \tau_R(\varphi_{D})$ using \ref{prop-torsion-ii}. As $C'_*$ has one less nontrivial homology group, the induction hypothesis gives $\tau_R(\varphi_{C'}) = \sum_{i > 0} (-1)^i [H_i(\varphi)]$, so it remains to show $\tau_R(\varphi_D) = [H_0(\varphi)]$. To do so, we consider the subcomplex $D'_* \subseteq D_*$ which is $\ker(d\colon D_0\ra D_{-1}) = H_0(C_*)$ in degree $*=0$ and $0$  otherwise otherwise. Once more we get an exact sequence $0 \to D'_* \to D_* \to D_*/D'_* \to 0$ and a quasi-isomorphism of short exact sequences induced by $\varphi_D$, so $\tau_R(\varphi_D) = \tau_R(\varphi_{D'}) + \tau_R(\varphi_{D_*/D'_*})$ using \ref{prop-torsion-ii}. Moreover, we have $\tau_R(\varphi_{D'})=[H_0(\varphi)]$ using \ref{prop-torsion-iv} and $\tau_R(D_*/D'_*)=0$ using \ref{prop-torsion-iv}, so the induction is concluded.

We now turn to the proof of the statement. The top map in \eqref{equ:rational-wh-torsion-square} is the composition
\[
\pi_0(\hAut^{B\pi}_\ast(X))\xleftarrow{\cong}\pi_0(\hAut^{B\pi,\mr{cell}}_\ast(X))\ra\pi_0(\hAut_{\bZ[\pi]}(C_\ast^{\mr{cell}}(\widetilde{X})))\xrightarrow{\tau_{\bZ[\pi]}}\widetilde{K}_1(\bZ[\pi])\twoheadrightarrow\Wh(\pi)
\]
where $\pi_0(\hAut^{B\pi,\mr{cell}}_\ast(X))$ is the group of cellular homotopy classes of cellular pointed homotopy equivalences over $B\pi$ and $\pi_0(\hAut_{\bZ[\pi]}(C_\ast^{\mr{cell}}(\widetilde{X})))$ is the group of chain homotopy classes of quasi-isomorphisms of the cellular chain complex over $\bQ[\pi]$ of the universal cover. (Note on passing that dividing out the subgroup $\langle \pm\pi\rangle\le K_1(\bZ[\pi])$ is not necessary to define the Whitehead torsion of a \emph{self}-equivalence, by the above discussion.) Next, we consider the diagram
\begin{equation}
\begin{tikzcd}
\pi_0(\hAut_{\bZ[\pi]}(C_\ast^{\mr{cell}}(\widetilde{X})))\rar{\tau_{\bZ[\pi]}}\dar{(-)\otimes \bQ}&\widetilde{K}_1(\bZ[\pi])\dar{(-)_\bQ}\rar& \Wh(\pi)\dar{(-)_\bQ}\\
\pi_0(\hAut_{\bQ[\pi]}(C_\ast^{\mr{cell}}(\widetilde{X})\otimes\bQ))\rar{\tau_{\bQ[\pi]}}\rar&\widetilde{K}_1(\bQ[\pi])\rar& \Wh(\pi)_\bQ
\end{tikzcd}
\end{equation}
which commutes since the $R$-torsion $\tau_R(-)$ is natural under base change (see the defining formula in loc.cit.). As $\bQ[\pi]$ is semisimple, the discussion above implies that the composition $\pi_0(\hAut_{\bZ[\pi]}(C_\ast^{\mr{cell}}(\widetilde{X})))\ra \Wh(\pi)_\bQ$ agrees with the composition
\[\textstyle{\pi_0(\hAut_{\bZ[\pi]}(C_\ast^{\mr{cell}}(\widetilde{X})))\lra \bigsqcap_{i=0}^{d}\Aut_{\bQ[\pi]}(H_i(\widetilde{X};\bQ))\xrightarrow{\sum_{i=0}^d(-1)^i[-]}\widetilde{K}_1(\bQ[\pi])\longtwoheadrightarrow \Wh(\pi)_\bQ}.\]
In conclusion, this shows that the clockwise composition of \eqref{equ:rational-wh-torsion-square} agrees with the composition
\[\textstyle{\pi_0(\hAut^{B\pi}_\ast(X))\lra \bigsqcap_{i=0}^{d}\Aut_{\bQ[\pi]}(H_i(\widetilde{X};\bQ))\xrightarrow{\sum_{i=0}^d(-1)^i[-]} \Wh(\pi)_\bQ}.\]
Factoring the first map as the composition
\[\textstyle{\pi_0(\hAut^{B\pi}_\ast(X))\xrightarrow{(-)_{\bQ}^\fib}\pi_0(\hAut^{B\pi}_\ast(X_\bQ^{\fib}))\lra \bigsqcap_{i=0}^{d}\Aut_{\bQ[\pi]}(H_i(\widetilde{X};\bQ))}\]
this is exactly the counterclockwise composition of \eqref{equ:rational-wh-torsion-square}, so the claim follows.
\end{proof}

\begin{remark}There is an alternative description of the map $\overline{\tau}$ in \eqref{equ:rational-wh-torsion-square}: view $K_1(\bQ[\pi])$ as the fundamental group of the algebraic $K$-theory (infinite loop) space of perfect chain complexes of $\bQ[\pi]$-modules and quasi-isomorphisms (see e.g.\,\cite[II.9.7.5, V.2.7.2]{Weibel}). Then given a self-equivalence representing a class in $\pi_0(\hAut_*^{B\pi}(\smash{X^\fib_\bQ}))$, lift it to the universal cover and take singular chains to obtain a self-equivalence of the perfect $\bQ[\pi]$-chain complex $C_\ast( \smash{\widetilde{X}^\fib_\bQ};\bQ)$ of rational singular chains. This gives a loop in the algebraic $K$-theory space, so an element in $K_1(\bQ[\pi])$.  This construction satisfies the analogues of \ref{prop-torsion-i}--\ref{prop-torsion-iv}, which one can use to show that it indeed agrees with the map $\overline{\tau}$.
\end{remark}

\begin{remark}\label{rem:short-cut}There is an alternative proof of \cref{thm:simple-htpy-mcg} in the case where $X$ is an orientable Poincaré complex of even formal dimension $d$, by showing that under this additional assumption the inclusion $\pi_0(\hAut^s(X))\le \pi_0(\hAut(X))$ has in fact finite index. To prove this, it suffices to show that $\pi_0(\hAut_*^{s,\pi_1}(X)) \subset \pi_0(\hAut_*^{\pi_1}(X))$ has finite index for which is it enough (as $\Wh(\pi_1(X))$ is finitely generated) to show that $\tau(f)\in  \Wh(\pi_1(X))$ is a torsion element for all $[f]\in\pi_0(\hAut_*^{\pi_1}(M))$. This is proved in  \cite[Proposition 7.2]{WallNorms}, by the following argument: the duality formula for Whitehead torsion gives $\tau(f)=(-1)^{d+1}\overline{\tau(f)}$ where $\overline{(-)}$ is the involution induced by transposing matrices and sending $g\in \pi_1(M)$ to $w(g)\cdot g^{-1}$ where $w\colon \pi_1(X)\ra \{\pm1\}$ is the orientation character. If $d$ is even, we thus have $2\cdot \tau(f)=\tau(f)-\overline{\tau(f)}$ which is a torsion-element as long as $w$ is trivial (which is equivalent to $X$ being orientable) by \cite[p.\,611]{WallNorms}. So $\tau(f)$ is a torsion element itself and the claim follows.
\end{remark}

\subsection{Finiteness properties of mapping class groups}\label{sec:finiteness-mcg}Building on his result for $\pi_0(\hAut(X))$, Sullivan \cite[Theorem 13.3]{Sullivan} also proved that the mapping class group $\pi_0(\Diff(M))$ of a smooth closed 1-connected manifold of dimension $d\ge6$ is commensurable up to finite kernel to an arithmetic group, so in particular of finite type. As mentioned in the introduction, in \cite[Corollary 5.3]{Triantafillou}, Triantafillou claims that $\pi_0(\Diff(M))$ is also of finite type for whenever $M$ is orientable, of dimension $d\ge5$, and has finite fundamental group. There appear to be two issues with the proposed proof of this in \cite{Triantafillou,TriantafillouhAut}, the first more critical than the second:

\begin{enumerate}[label=(\alph*)]
\item\label{issue-1} The first issue is that the proof crucially relies on the claim \cite[Proposition 15]{TriantafillouhAut} that for all $M$ as above, the image of $\pi_0(\Diff(M))$ in $\pi_0(\hAut(M))$ is finite index in the stabiliser $\pi_0(\hAut(M))_{[T^sM]}\le  \pi_0(\hAut(M))$ of the stable tangent classifier $[T^sM]\in[M,BO]$. This is incorrect, see \cref{exam:counterexample} below for counterexamples in all odd dimensions $\geq 5$. The issue with the proof is that Triantafillou implicitly assumes that the stabiliser of the identity under the action of $\pi_0(\hAut(M))_{[T^sM]}$ on the non-simple version of the structure set $\mathcal{S}^h(M)$ is the subgroup of those equivalences that are homotopic to diffeomorphisms. However, this stabiliser agrees instead with the (often larger) subgroup of those tangential self-equivalences of $M$ that are homotopic to a self-equivalence induced by an inertial $h$-cobordism on $M$. 

\item \label{issue-2}The second issue concerns the claim in dimension 5. The proof of \cite[Corollary 5.3]{Triantafillou} relies on work of Hatcher--Wagoner \cite{HatcherWagoner} and Igusa \cite{Igusa} which assumes the dimension to be at least $6$. Hatcher \cite[p.\,7]{Hatcher} states that the necessary statement was extended to $d=5$ by Igusa, but to our knowledge no proof has appeared so far.
\end{enumerate}

Using \cref{thm:simple-htpy-mcg}, we are able to circumvent the first issue, basically by replacing the role of the action of $\pi_0(\hAut(M))_{[T^sM]}$ on $\mathcal{S}^h(M)$ in Triantafillou's argument by the action of $\pi_0(\hAut^s(M))_{[T^sM]}$ on $\mathcal{S}^s_(M)$. As a courtesy to the reader, we spell out the argument in full detail, which also gives us the chance to extend the result to homeomorphisms and to nonorientable manifolds. The following in particular proves \cref{thm:mcg} \ref{mainthm-mcg-i}.

\begin{theorem}\label{thm:mcg-finfty} Let $M$ be a closed smooth manifold of dimension $d\geq 6$ with finite fundamental group at all basepoints. Then the groups $\pi_0(\Diff(M))$ and $\pi_0(\Homeo(M))$ are of finite type and have polycyclic solvable subgroups in the sense of \cref{def:polyclic-solvable-subgroups}. \end{theorem}

\begin{proof}Directly from the definitions, we see that there is a commutative diagram
\begin{equation}\label{equ:diff-blockdiff}
\begin{tikzcd}
\pi_0(C^{\Diff}(M))\rar\dar& \pi_0(\Diff(M))\rar\dar& \pi_0(\widetilde{\Diff}(M))\dar\rar& 0\\[-4pt]
\pi_0(C^{\mr{Top}}(M)) \rar& \pi_0(\Homeo(M))\rar& \pi_0(\widetilde{\Homeo}(M))\rar& 0
\end{tikzcd}
\end{equation}
with exact rows. Here the leftmost groups are the path-components of the groups $C^{\Diff}(M)$ and $C^{\mr{Top}}(M)$ of concordance diffeomorphisms and homeomorphisms, the terms in the third column are the groups of pseudoisotopy classes of diffeomorphisms and homeomorphisms, and the vertical maps are the canonical forgetful maps (see e.g.~\cite{Hatcher} for definitions).

 We first argue that the leftmost vertical map is an isomorphism. For this it suffices to prove that the homotopy fibre $C^{\mr{Top}}(M)/C^{\Diff}(M)$ of the comparison map $BC^{\Diff}(M)\ra BC^{\mr{Top}}(M)$ is 1-connected. By smoothing theory (see \cite[p.\,453--455]{BurgheleaLashofSuspension}), this fibre agrees with a collection of path components of a space of sections $\mr{Sect}(E\ra M)$ of a fibration whose fibre is equivalent to the homotopy fibre of the stabilisation map
\[\hofib\big(\mr{Top}(d)/\mr{O}(d)\ra \mr{Top}(d+1)/\mr{O}(d+1)\big).\]
As this fibre is $(d+2)$-connected (see \cite[Essay V, \S 5, 5.0 (4)]{KirbySiebenmann}), it follows from obstruction theory that $\mr{Sect}(E\ra M)$ is $1$-connected, so the same holds for $C^{\mr{Top}}(M)/C^{\Diff}(M)$.

By surgery theory, the homotopy fibre of the map 
\begin{equation}\label{equ:block-comparison}
\BlockBDiff(M)\lra \BlockBHomeo(M)
\end{equation} 
is homotopy equivalent to a collection of components of the mapping space $\Map(M,\mr{Top}/\mr{O})$. As $\mr{Top}/\mr{O}$ has finite homotopy groups, it follows from obstruction theory that this mapping space has finitely many path components each of which has degreewise finite homotopy groups, so the same holds for the homotopy fibre of \eqref{equ:block-comparison}. In particular, the two rightmost groups in \eqref{equ:diff-blockdiff} are commensurable up to finite kernel. As $\pi_0(C^{\Diff}(M))\cong \pi_0(C^{\mr{Top}}(M))$ is abelian \cite[Lemma 1.1, p.\,18]{HatcherWagoner} and groups of finite type and with polycyclic solvable subgroups are closed under extensions and commensurability up to finite kernels (see Lemmas \ref{lem:f-infty} and \ref{lem:polycyclic-sovable-subgroups}), all we are left to do is to show that

\begin{enumerate}
\item \label{concordance-finite}$\pi_0(C^{\Diff}(M))$ is finitely generated and
\item \label{blockdiff-finite-type}$\pi_0(\smash{\widetilde{\Diff}}(M))$ is of finite type and has polycyclic solvable subgroups. 
\end{enumerate}
Writing $\pi_0(\hAut^{\cong}(M))\le \pi_0(\hAut^s(M))$ for the subgroup of those simple homotopy automorphisms that are homotopic to a diffeomorphism, Item \ref{blockdiff-finite-type} will follow from the exact sequence
\[\pi_1(\hAut^s(M)/\widetilde{\Diff}(M);\mr{id}) \lra \pi_0(\widetilde{\Diff}(M)) \lra \pi_0(\hAut^{\cong}(M))\lra 1\]
and another application of \cref{lem:f-infty} and \cref{lem:polycyclic-sovable-subgroups} together with the fact that polycyclic groups are closed under quotients (see \cref{lem:group-classes-closure}) once we show

\begin{enumerate}[resume]
\item \label{equ:polycyclic-quotient}$\pi_1(\hAut^s(M)/\widetilde{\Diff}(M);\mr{id})$ is polycyclic and
\item \label{equ:haut-image-arithmetic}$\pi_0(\hAut^{\cong}(M))$ is commensurable up to finite kernel with an arithmetic group.
\end{enumerate}

We now explain the proofs of \ref{concordance-finite} and \ref{equ:polycyclic-quotient}; claim \ref{equ:haut-image-arithmetic} will be proved separately as \cref{prop:diff-in-haut-arithmetic} below. For \ref{concordance-finite}, we use that since $d \geq 6$, there is an exact sequence of abelian groups \[\Wh^+_1(\pi_1(M);\bZ/2 \oplus \pi_2(M)) \lra \pi_0(C^{\Diff}(M)) \lra \Wh_2(\pi_1(M)) \lra 0\]
by \cite[p.\,10-11]{HatcherWagoner} and \cite[p.\,104-105]{Igusa}. By definition, the rightmost group is a quotient of $K_2(\bZ[\pi_1(M)])$ and the leftmost group a quotient of $(\bZ/2 \oplus \pi_2(M))[\pi_1(M)]$.  The latter is finitely generated by \cref{coro:pi1-finite-pik-fg} because $\pi_1(M)$ is finite and $\pi_2(M)$ is finitely generated using \cref{lem:fg-homology-implies-fg-homotopy-finite}. The former is finitely generated by \cite[Theorem 1.1.(i)]{Kuku}.

For the proof of \ref{equ:polycyclic-quotient}, we use the simple surgery exact sequence (see \cite[Theorem 10.8]{WallBook}) in the form of an exact sequence of groups
\[L^s_{d+2}(\bZ[\pi_1(M)],w) \lra \pi_1(\hAut^s(M)/\widetilde{\Diff}(M)) \lra [\Sigma M_+,\mr{G}/\mr{O}],\]
where $L^s_{\ast}(\bZ[\pi_1(M)],w)$ denotes the simple (quadratic) $L$-groups of $\pi_1(M)$ with respect to the orientation character $w\colon \pi_1(M)\ra \{\pm1\}$ (see the next paragraph for more details). We claim that the two outer terms in this sequence are finitely generated abelian, which would imply \ref{equ:polycyclic-quotient}. For $[\Sigma M_+,\mr{G}/\mr{O}]$, this holds as a consequence of \cref{lem:section-spaces-finite-CW} and the fact that the infinite loop space $\mr{G}/\mr{O}$ has finitely generated homotopy groups. For the groups $L^s_{\ast}(\bZ[\pi_1(M)],w)$, finite generation follows from \cite[Theorem 7.3]{WallHermitianV}, but this requires some unwrapping:

The notation Wall uses in \cite{WallHermitianV} is introduced in \cite{WallConference}: given an \emph{anti-structure} $R=(R,\alpha,u)$ (a unital ring $R$ together with an anti-automorphism $\alpha\colon R\ra R$ and a unit $u\in R^\times$ satisfying some conditions \cite[p.\,2]{WallConference}) and a subgroup $H\le K_1(R)$ that is invariant under the involution on $K_1(R)$ induced by taking transpose of matrices followed by applying $\alpha$ coefficient-wise (denoted $T$ in \cite{WallConference}, see page 16), Wall defines $L$-groups $L_\ast^H(R)$ \cite[p.\,20]{WallConference} (the dependence on $\alpha$ and $u$ is not reflected in his notation). In the case $R=\bZ[\pi_1(M)]$, $u=1$, $\alpha(g)=w(g)\cdot g^{-1}$, and $H$ the image of $\pm\pi_1(M)$ in $K_1(R)$, the groups denoted $L^s_{*}(\bZ[\pi_1(M)],w)$ above are Wall's $L_{*}^H(R)$ \cite[p.\,24-25]{WallConference}. In \cite[Theorem 7.3]{WallHermitianV} he proves that $L^H_\ast(R)$ is degreewise finitely generated for a certain choice of invariant subgroup $H\le K_1(R)$ and any anti-structure $(R,\alpha,u)$ such that the additive group of $R$ is finitely generated (this holds in our case $R=\bZ[\pi_1(M)]$ since $\pi_1(M)$ is finite) and that the $\bQ$-algebra $R\otimes\bQ$ is semisimple (this holds in our case by Maschke's theorem). There are no restrictions on $\alpha$ or $u$. This only implies that $L^H_\ast(\bZ[\pi_1(M)])$ is degreewise finitely generated for a specific choice of $H$, but as Wall indicates on page 286 of loc.cit., this is enough to conclude finite-generation for any other choice of $H$: one applies a form of the Rothenberg sequence \cite[Theorem 3, p.\,21]{WallConference} which compares $L_\ast^H(\bZ[\pi_1(M)])$ with $L_\ast^{H'}(\bZ[\pi_1(M)])$ for involution invariant subgroups $H\le H'$ of $K_1(\bZ[\pi_1(M)])$. The relative terms are the Tate cohomology groups of the induced involution on $H'/H$, so they are finite since they are $2$-torsion by definition and finitely generated as subquotients of the finitely generated group $K_1(\bZ[\pi_1(M)])$.\end{proof}

\begin{proposition}\label{prop:diff-in-haut-arithmetic}Let $d\ge6$ and $M$ a closed smooth $d$-manifold with finite fundamental group at all basepoints. The group $\pi_0(\hAut^{\cong}(M))$ of homotopy classes of self-equivalences that are homotopic to a diffeomorphism is commensurable up to finite kernel with an arithmetic group.
\end{proposition}
%As $M$ is orientable, we may show instead of \ref{equ:haut-image-arithmetic} that the image of the finite-index subgroup $\pi_0(\Diff^+(M))\le \pi_0(\Diff(M))$ of orientation preserving diffeomorphisms in the finite-index subgroup $\pi_0(\hAut^+(M))\le \pi_0(\hAut(M))$ of orientation preserving homotopy self-equivalences has finite index. 

\begin{proof}
Since every diffeomorphism is simple and preserves the tangent bundle, the group $\pi_0(\hAut^{\cong}(M))$ is contained in the subgroup $\pi_0(\hAut^s(M))_{[T^sM]}\le  \pi_0(\hAut^s(M))$ of simple self-equivalences that stabilise the stable tangent classifier $[T^sM]\in[M,B\mr{O}]$ with respect to the $\pi_0(\hAut(M))$-action by precomposition. We first argue that the group $\pi_0(\hAut^s(M))_{[T^sM]}$ is commensurable up to finite kernel to an arithmetic group. 

Arguing as in the first part of the proof of \cref{thm:simple-htpy-mcg}, we may, firstly, assume that $M$ is connected and, secondly, replace the group $\pi_0(\hAut^s(M))_{[T^sM]}$ by its analogue $\pi_0(\hAut^{s,\pi_1}_\ast(M))_{[T^sM]}$ where we require the equivalences additionally to preserve a basepoint and to induce the identity on fundamental groups. The group $\pi_0(\hAut^{s,\pi_1}_\ast(M))_{[T^sM]}$ is commensurable up to finite kernel to an arithmetic group by the second part of \cref{thm:simple-htpy-mcg} since it is commensurable up to finite kernel to the stabiliser of the image $[T^sM_\bQ]$ of $[T^sM]$ under the map $[M,B\mr{O}]\ra [M,B\mr{O}_\bQ]$ induced by rationalising the simple space $B\mr{O}$ (which has finite preimages by obstruction theory). This uses that $B\mr{O}_\bQ$ is a product of Eilenberg--MacLane spaces, so stabilising $[T^sM_\bQ]$ is equivalent to stabilising a finite set of cohomology classes.

It thus suffices to show that  $\pi_0(\hAut^{\cong}(M))\le \pi_0(\hAut^s(M))_{[T^sM]}$ has finite index. To do this, we use the simple surgery exact sequence from the previous proof, but this time the part
\begin{equation}\label{equ:surgery-sequence-sets}
	L^s_{d+1}(\bZ[\pi_1(M)],w)\xlra{\omega} \cS^s(M)\xlra{\eta} [M,\mr{G}/\mr{O}].
\end{equation}
Several properties of this sequence will play a role:
\begin{enumerate}[label=(\alph*)]
\item\label{surg-prop-a} The outer terms are abelian groups (for $[M,\mr{G}/\mr{O}]$ this uses the infinite loop space structure on $\mr{G}/\mr{O}$). The simple structure set $\cS^s(M)$ (the set of homotopy classes of simple homotopy equivalences from closed $d$-manifolds to $M$, up to precomposition with diffeomorphisms) is a pointed set with basepoint $[\mr{id}_M]\in \cS^s(M)$.
\item\label{surg-prop-b}  The sequence is exact as a sequence of pointed sets (see \cite[Theorems 10.3, 10.5]{WallBook}).
\item\label{surg-prop-c}  The group $L^s_{d+1}(\bZ[\pi_1(M)],w)$ acts on $\cS^s(M)$ (in the category of \emph{unpointed} sets) such that two elements of $\cS^s(M)$ have the same image under $\eta$ if and only if they lie in the same orbit (see Theorem 10.5 loc.cit.). The map $\omega$ is given by acting on $[\mr{id}_M]\in\cS^s(M)$.
\item\label{surg-prop-d}  The map $\eta$ is $\pi_0(\hAut^s(M))$-equivariant with respect to appropriate actions on source and target (in the category of \emph{unpointed} sets): an equivalence $[\phi]\in \pi_0(\hAut^s(M))$ acts on $\cS^s(M)$ by postcomposition and on $[M,\mr{G}/\mr{O}]$ by sending $[f]\in [M,\mr{G}/\mr{O}]$ to $[f \circ \overline{\phi}]+\eta(\phi)$ where $\overline{\phi}$ is a homotopy inverse to $\phi$ (see e.g.\,\cite[Lemma\,3.3]{BerglundMadsen}).
\item\label{surg-prop-e}  Writing $j\colon \mr{G}/\mr{O}\ra B\mr{O}$ for the usual map, the composition 
\begin{equation}\label{equ:comp-surg}\cS^s(M)\xlra{\eta} [M,\mr{G}/\mr{O}]\xlra{j^*} [M,B\mr{O}]\end{equation}
sends $[f\colon N\ra M]\in \cS^s(M)$ to $(T^sM-\overline{f}^*(T^sN))\in [M,B\mr{O}]$ (cf.\,\cite[p.\,113-114]{WallBook}).
\end{enumerate}
Equipping $[M,B\mr{O}]$ with the $\pi_0(\hAut^s(M))$-action defined by the analogue of the formula in \ref{surg-prop-d} using the composition \eqref{equ:comp-surg} in place of $\eta$ (note that this action is \emph{not} the one by precomposition), the sequence \eqref{equ:comp-surg} consists of $\pi_0(\hAut^s(M))$-equivariant maps. The final map in \eqref{equ:comp-surg} has finite preimages by obstruction theory since $G$ has finite homotopy groups, so the stabiliser $\pi_0(\hAut^s(M))_\eta\le \pi_0(\hAut^s(M))$ of the basepoint in $[M,\mr{G}/\mr{O}]$ (the constant map) has finite index in the stabiliser of the basepoint in $[M,B\mr{O}]$. The latter agrees with $\smash{\pi_0(\hAut^s(M))_{[T^sM]}}$ as a result of  \ref{surg-prop-e}. It thus suffices to show that stabiliser $\pi_0(\hAut^{\cong}(M))$ of the basepoint $[\mr{id}_M]\in\cS^s(M)$ has finite index in the stabiliser $\pi_0(\hAut^s(M))_\eta$ of the basepoint in $[M,\mr{G}/\mr{O}]$, i.e.\,that the set $\pi_0(\hAut^s(M))_\eta/\pi_0(\hAut^{\cong}(M))$ of cosets is finite. By equivariance of $\eta$ and the fact that $\mr{ker}(\eta)=\mr{im}(\omega)$ by \ref{surg-prop-b}, the latter agrees with
\begin{equation}\label{equ:intersected-structure-set}
\Big(\pi_0(\hAut^s(M))/\pi_0(\hAut^{\cong}(M))\Big)\cap \mr{im}(\omega)\subset \cS^s(M),
\end{equation}
where we view the left-hand quotient as a subset of $\cS^s(M)$ via the action of $\pi_0(\hAut^s(M))$ on the basepoint $[\mr{id}_M]\in\cS^s(M)$, with stabiliser $\pi_0(\hAut^{\cong}(M))$. This leaves us with showing that the intersection \eqref{equ:intersected-structure-set} is finite. The argument for this depends on the parity of the dimension.

For even $d$, in fact the whole image $\mr{im}(\omega)$ of $\omega\colon L^s_{d+1}(\bZ[\pi_1(M)],w)\ra \cS^s(M) $ is finite, since the entire group $L^s_{d+1}(\bZ[\pi_1(M)],w)$ is finite: by another application of the Rothenberg sequence whose relative terms are finite (see the proof of \cref{thm:mcg-finfty}), it suffices to show finiteness for the variant of $L_{d+1}^s(\mathbb{Z}[\pi_1(M)],w)$ for any involution-invariant subgroup of $K_1(\mathbb{Z}[\pi_1(M)])$. For a certain choice of subgroup this is stated in \cite[p.\,2 (4)]{WallHermitianVI} with a reference to \cite{WallHermitianV}, but since it is cumbersome to extract the statement from the latter, we refer to \cite[Corollary 2]{Hausmann} instead which shows that the variant for the trivial subgroup $H=K_1(\mathbb{Z}[\pi_1(M)])$ is annihilated by $8$, so is finite by finite-generation (see the end of the proof of \cref{thm:mcg-finfty}).

For odd $d$, we consider the subgroup
$\Lambda \le L^s_{d+1}(\bZ[\pi_1(M)],w)$ of elements that are torsion in the cokernel of the map $\iota\colon L^s_{d+1}(\bZ)\ra L^s_{d+1}(\bZ[\pi_1(M)],w)$ induced by the inclusion of rings with anti-automorphism $\bZ\le \bZ[\pi_1(M)]$ where we equip $\bZ$ with the trivial anti-automorphism. We will finish the proof by showing that \eqref{equ:intersected-structure-set} is finite in two steps:
\begin{enumerate}
\item\label{omega-lambda-finite} The image $\omega(\Lambda)$ of the subgroup $\Lambda$ under $\omega$ is finite and
\item\label{smaller intersection} \eqref{equ:intersected-structure-set} agrees with the a priori smaller intersection $\big(\pi_0(\hAut^s(M))/\pi_0(\hAut^{\cong}(M))\big)\cap \omega(\Lambda)$.
\end{enumerate}

To see \ref{omega-lambda-finite}, note that the image of $\iota\colon L^s_{d+1}(\bZ)\ra L^s_{d+1}(\bZ[\pi_1(M)],w)$ has finite index in $\Lambda$ since by construction the quotient $\Lambda/\mr{im}(\iota)$ is isomorphic to the torsion subgroup of the finitely generated group $L^s_{d+1}(\bZ[\pi_1(M)],w)/\mr{im}(\iota)$. As $\omega$ is given by acting on $[\mr{id}_M]\in \cS^s(M)$ with respect to the action of $L^s_{d+1}(\bZ[\pi_1(M)],w)$ on $\cS^s(M)$, it thus suffices to show that the potentially smaller image $\mr{im}(\omega\circ\iota)$ is finite which in turn follows from the naturality of the surgery exact sequence by choosing an embedded disc $D^d\subset M$, and the fact that $\cS^s_\partial(D^d)$ is finite as a result of the classification of exotic spheres. This proves \ref{omega-lambda-finite}.

To prove \ref{smaller intersection}, by definition of the subgroup $\Lambda$, it suffices to show that for a given class $x\in L^s_{d+1}(\bZ[\pi_1(M)],w)$ that maps under $\omega$ to \eqref{equ:intersected-structure-set}, there is a multiple $n\cdot x$ with $n\neq0$ that lies in the image of the map $\iota\colon L^s_{d+1}(\bZ)\ra L^s_{d+1}(\bZ[\pi_1(M)],w)$. It is enough to test this property after applying a morphism with domain $L^s_{d+1}(\bZ[\pi_1(M)],w)$ and finite kernel. The morphism $\smash{\sigma_\bZ\colon L^s_{d+1}(\bZ[\pi_1(M)],w)\ra R_{\bC}(\pi^+_1(M))}$ we shall use has target the complex representation ring of the kernel $\pi_1^+(M)\coloneqq\ker(w\colon \pi_1(M)\ra \{\pm1\})$. It is given as the composition 
\begin{equation}\label{equ:multi-signature-comp}
	L^s_{d+1}(\bZ[\pi_1(M)],w)\lra L^s_{d+1}(\bR[\pi_1(M)],w)\xlra{\sigma_\bR}R_{\bC}(\pi^+_1(M))
\end{equation}
of the map induced by the inclusion $\bZ\le \bR$ and the \emph{multi-signature} $\sigma_\bR$ (see \cite[Section 2.2]{WallHermitianVI}). This composition indeed has finite kernel, which follows by combining (i) that $L^s_{d+1}(\bZ[\pi_1(M)],w)$ is finitely generated (see the end of the proof of \cref{thm:mcg-finfty}), (ii) that the first map in \eqref{equ:multi-signature-comp} has finite kernel by \cite[Theorem 7.3]{WallHermitianV} (note that this result is phrased for a differently decorated $L$-group, but this implies what we need by the Rothenberg sequence, as in the proof of \cref{thm:mcg-finfty}), and (iii) that the kernel of the second map in \eqref{equ:multi-signature-comp} is torsion (see \cite[Theorem 2.2.1]{WallHermitianVI}; note that the result in the non-orientable case is on page 22). 

We next identify the image $\sigma_{\bZ}(x)\in R_{\bC}(\pi^+_1(M))$ of the given class $x\in L^s_{d+1}(\bZ[\pi_1(M)],w)$. To begin with, note that since the class $x$ maps to the intersection \eqref{equ:intersected-structure-set} under $\omega$, it can be represented as the surgery obstruction of a degree $1$ normal map $W\ra M\times [0,1]$ of \emph{self}-bordisms of $M$ which is the identity on one end and a simple homotopy equivalence on the other (We pass between $L$-groups and surgery problems as in \cite[§5]{WallBook} which involves the choice of orientation of the universal cover $\widetilde{M}$; see page 46 loc.cit.. We make this choice once and for all.). In these terms, the class $\sigma_\bZ(x)$ is given (see \cite[Section 13 B]{WallBook}) as a difference \[\sigma_\bZ(x)=\sigma(\pi_1^+(M),\widetilde{W})-\sigma(\pi_1^+(M),\widetilde{M}\times I)\in R_{\bC}(\pi^+_1(M))\] where $\sigma(G,N)$, for an oriented compact manifold $N$ with orientation-preserving action by a finite group, is the \emph{$G$-signature} in the sense of \cite[p.\,578-579, 587-588]{AtiyahSingerIII}. Here the orientations of $\smash{\widetilde{W}}$ and $\smash{\widetilde{M}}\times I$ are induced by the chosen orientation of $\smash{\widetilde{M}}$. The second summand $\sigma(\pi_1^+(M),\widetilde{M}\times I)$ in fact vanishes; this follows from the definition on p.\ 588 loc.cit. since the map $\varphi$ on the top of that page is zero for $\smash{\widetilde{M}\times I}$ because the inclusion $\smash{\widetilde{M}\times \{0,1\}\subset \widetilde{M}\times I}$ has a homotopy section.

To finish the proof of \ref{smaller intersection}, we are thus left to show that some multiple $n\cdot \sigma(\pi_1^+(M),\widetilde{W})\in R_{\bC}(\pi^+_1(M))$ with $n\neq 0$ lies in the image of the composition $(\sigma_\bZ\circ\iota)\colon L^s_{d+1}(\bZ)\ra R_{\bC}(\pi^+_1(M))$. For this, we use that the bordism groups $\Omega^{\mr{SO}}_\ast(BG)$ of oriented $d$-manifolds with a free action of a finite group $G$ are finite in odd degrees by an application of the Atiyah--Hirzebruch spectral sequence, so there exists a positive integer $m>0$ such that $\smash{\sqcup^m\widetilde{M}}$ bounds $\pi_1^+(M)$-equivariantly a manifold $P$ with free action. Gluing $P$ to both ends of $\smash{\sqcup^m\widetilde{W}}$, we obtain a closed oriented $\pi_1^+(M)$-manifold $V$, still with free action. Moreover, we have 
\[\sigma(\pi_1^+(M), V)=\sigma(\pi_1^+(M), P)+\sigma(\pi_1^+(M),\sqcup^m\widetilde{W})+\sigma(\pi_1^+(M),-P)=m\cdot \sigma(\pi_1^+(M),\widetilde{W})\]
 by the additivity of the $G$-signature \cite[Proposition (7.1)]{AtiyahSingerIII} and the fact that taking the opposite orientation (indicated by a minus-sign) negates the $G$-signature. As $m\neq0$ it thus suffices to show $n\cdot \sigma(\pi_1^+(M),V)\in \mr{im}(\sigma_\bZ\circ\iota)$ for some $n\neq0$, which is what we do next.
  
Since $V$ is a \emph{closed} manifold, the class $\sigma(\pi_1^+(M), V)\in R_\bC(\pi_1^+(M))$ is a multiple of the regular representation (see \cite[Proposition 13B.1]{WallBook} or specialise \cite[(6.12)]{AtiyahSingerIII} to free actions; this uses that a class in $R_\bC(\pi_1^+(M))$ is determined by its character.). To use this to prove the remaining claim, we distinguish some cases. If $d+1 \equiv 2 \pmod 4$ or when the orientation character $w$ is nontrivial, we show that the only multiple of the regular representation in the image of $\sigma_{\bZ}$ is $0$, which then implies that $\sigma(\pi_1^+(M), V)$ is trivial, so in particular in the image of $(\sigma_\bZ\circ\iota)$. For this, we use that a class in $R_\bC(\pi_1^+(M))$ is determined by its character, and read off from \cite[Theorem 2.2.1]{WallHermitianVI} which characters are realised by classes in the image of $\sigma_\bZ$. Firstly, if $d+1 \equiv 2 \pmod 4$, then all characters in the image take values in $i\cdot\bR\subset \bC$, which is not the case for any nontrivial multiple of the regular representation (consider the value at $1\in\pi_1^+(M) $). Secondly, if $w$ is non-trivial then characters in the image of $\sigma_\bZ$ satisfy a certain conjugation condition (see Theorem 2.2.1 loc.\,cit.) which is again not satisfied for any nontrivial multiple of the regular representation (as before, consider the value at $1\in\pi_1^+(M) $).

It remains to consider the case where $d+1 \equiv 0 \pmod 4$ and $w$ is trivial. Directly from the definition of $\sigma_\bZ$ in \cite[Section 2.2]{WallHermitianVI} we see that the image $\mr{im}(\sigma_\bZ\circ\iota)\subset R_{\bC}(\pi^+_1(M))$ is contained in the cyclic subgroup $\langle R\,\rangle$ spanned by the regular representation. If we knew that $\sigma_\bZ\circ\iota$ is injective, then since $L^s_{d+1}(\bZ)\cong \bZ$, the inclusion $\mr{im}(\sigma_\bZ\circ\iota)\le \langle R\,\rangle$ would have finite index, so since $\sigma(\pi_1^+(M), V)\in \langle R\,\rangle$ we would have $\smash{n\cdot \sigma(\pi_1^+(M), V)\in \mr{im}(\sigma_\bZ\circ\iota)}$ for some $n\neq0$, as claimed. To prove injectivity of $\sigma_\bZ\circ\iota$, note that since $w$ is trivial, the map $ L^s_{d+1}(\bZ)\ra L^s_{d+1}(\bZ[\pi_1(M)],w)$ is split injective (consider the augmentation $\bZ[\pi_1(M)]\ra \bZ$ which is a morphism of rings with anti-automorphism), so $\sigma_\bZ\circ\iota$ has to be injective as $\sigma_\bZ$ has finite kernel and $\iota$ has domain $L^s_{d+1}(\bZ)\cong \bZ$. This concludes the proof.
\end{proof}

\begin{remark}\label{rem:relative-arithm}
 \label{relative-arit}To our knowledge, the analogue of \cref{thm:mcg-finfty} for manifolds with nonempty boundary and diffeomorphisms fixing the boundary has not been proved yet, though there is a version in the simply connected case \cite{KupersMCG}. Such a relative finiteness result would simplify several steps in the proof of our main \cref{thm:fg pi}.
\end{remark}

\begin{example}\label{exam:counterexample} We now construct counterexamples in every odd dimension $d\ge5$ to the claim in \cite[Proposition 15]{TriantafillouhAut} and \cite[Theorem 5.1]{Triantafillou} that the image of the group $\pi_0(\Diff(M))$ in $\pi_0(\hAut(M))_{[T^sM]}$ has finite index for all closed smooth orientable manifolds of dimension $d\ge5$ with finite fundamental group, by showing that the inclusion $\pi_0(\hAut^s(M))_{[T^sM]} \leq \pi_0(\hAut(M))_{[T^sM]}$ is not of finite index. This invalidates the proofs of \cite[Theorems 3--4, Proposition 15, Corollary 16]{TriantafillouhAut} and \cite[Theorem 5.1--5.2, 5.4, Corollary 5.3]{Triantafillou}.
	
First we explain how ideas of Lawson \cite{LawsonFinite,LawsonTrivializing} lead to closed connected smooth manifolds $M$ of dimension $d$ with the following properties:
	\begin{enumerate}
		\item\label{counter-i} $M$ is stably parallelisable.
		\item\label{counter-ii} Every element $\tau\in \Wh(\pi_1(M))$ is realised by an inertial $h$-cobordism $W_\tau \colon M \leadsto M$ such that the induced automorphism of $\pi_1(M)$ is the identity.
	\end{enumerate}
For this, fix a finitely-presented group $\pi$ for which the map $\mathrm{GL}_n(\mathbb{Z}[\pi]) \to \Wh(\pi)$ is surjective for some $n$, and a based connected finite $2$-complex $X$ with $\pi_1(X)=\pi$. Choose an embedding $X \hookrightarrow \mathbb{R}^{d+1}$ (which exists for $d \geq 4$) and a regular neighbourhood $N$ with a retraction $r \colon N \to X$ as the homotopy inverse of the inclusion $\mathrm{inc}\colon X \hookrightarrow N$. Choosing $M \coloneqq \partial N$, the first condition is satisfied since the $d$-manifold $M$ embeds by construction into $\bR^{d+1}$ with trivial normal bundle. To show \ref{counter-ii}, use that any $\tau \in \Wh(\pi)$ arises as the torsion of a self-equivalence $f_\tau \colon X \vee \bigvee_n S^2 \to X \vee \bigvee_n S^2$ that is the identity on $X$, so it in particular induces the identity on $\pi$ (represent $\tau$ by a matrix $A\in \mathrm{GL}_n(\mathbb{Z}[\pi])$,  use it as instruction of how to construct a $\pi$-equivariant self-equivalence of the universal cover, and take the quotient by $\pi$). The composition $(\mathrm{inc} \circ f_\tau \circ r) \colon N \to N$ is homotopic to an embedding $e_\tau \colon N \hookrightarrow \interior(N)$ (make it an immersion by Smale--Hirsch theory and then an embedding by general position) and its complement $W_\tau = N \setminus e_\tau(\interior(N))$ is the inertial $h$-cobordism $M\leadsto M$ sought after  (use excision for \ref{counter-ii}). 

If $\pi$ is finite, then $\mathrm{GL}_2(\mathbb{Z}[\pi]) \to \Wh(\pi)$ is surjective \cite[V, §4]{Bass}, so the above construction applies to give a $d$-manifold $M$ for any $d\ge 4$ with $\pi_1(M)=\pi$ and properties \ref{counter-i} and \ref{counter-ii}. For $\tau\in \Wh(\pi)$, we pick $W_\tau$ as in \ref{counter-ii} and consider the induced self-equivalence $g_\tau \colon M \to M$. As $M$ is stably parallelisable, $g_\tau$ clearly stabilises  $[T^sM] \in [M,B\rm{O}]$. The Whitehead torsion of $g_\tau$ is given by $\tau+(-1)^{d+1}\overline{\tau}\in \Wh(\pi)$ (use the composition and duality formula for Whitehead torsions \cite[Lemma 7.8, §10]{Milnor}), so since the involution on $\Wh(\pi)$ is trivial after passing to the maximal torsion-free quotient $\Wh(\pi)/\mathrm{tors}$ (see \cite[Corollary 6.10]{Milnor} or \cite[p.\,611]{WallNorms}), the torsion of $g_\tau$ is $2\cdot[\tau]\in \Wh(\pi)/\mathrm{tors}$ if $d$ is odd. Hence as long as $ \Wh(\pi)/\mathrm{tors}$ is non-trivial, the subgroup $\pi_0(\hAut^s(M))_{[T^sM]} \leq \pi_0(\hAut(M))_{[T^sM]}$ has infinite index. Examples of finite groups $\pi$ with non-trivial $\Wh(\pi)/\mathrm{tors}$ abound, e.g.~$\mathbb{Z}/5\mathbb{Z}$ will do \cite[Example 6.6]{Milnor}.
\end{example}

\section{Finiteness properties of stable homology groups}\label{section:GRW}
A key ingredient in the proof of \cref{thm:fg pi} is Galatius--Randal-Williams' work  \cite{GRWstable,GRWII,GRWI} on moduli spaces of manifolds and its extension by Friedrich \cite{Friedrich} to certain nontrivial fundamental groups. This section serves to use their work to deduce finiteness results for the homology of $\BDiff(M)$ for certain $M$ in a range.

In order to state a form of their results suitable for our purposes, we fix a closed manifold $M$ of even dimension $2n$ and a factorisation over a connected space $B$
\begin{equation}\label{equ:tangential-structure}
M\xlra{\ell_M} B\xlra{\lambda}B\mr{O}(2n)
\end{equation}
of a map $M\ra B\mr{O}(2n)$ classifying the tangent bundle of $M$. The result by the aforementioned authors we are about to state concerns the homotopy quotient
\begin{equation}\label{equ:homotopy-quotient}
\BDiff^\lambda(M)\coloneqq \mr{Bun}(TM,\lambda^*\gamma_{2n}) \sslash\Diff(M)
\end{equation}
of the action via the derivative of $\Diff(M)$ on the space of bundle maps from the tangent bundle $TM$ to the pullback $\lambda^*\gamma_{2n}$ along $\lambda$ of the universal $2n$-plane bundle $\gamma_{2n}\ra B\mr{O}(2n)$. The map $\ell_{M}$ is covered by such a bundle map $\ell_M\colon TM\ra \lambda^*\gamma_{2n}$, denoted by the same symbol. This determines a path component of \eqref{equ:homotopy-quotient}, which we denote by $\BDiff^\lambda(M)_{\ell_M}\subset \BDiff^\lambda(M)$. Furthermore, we fix a Moore--Postnikov $n$-factorisation
\begin{equation}\label{equ:factorisation-tangential-structure}
M\xlra{\rho_M}B'\xlra{u}B
\end{equation}
of $\ell_M$, i.e.\,a factorisation into an $n$-connected cofibration $\rho_M$ followed by an $n$-co-connected fibration $u$. We abbreviate $\theta\coloneqq (\lambda\circ u)$ and denote by $\mr{MT}\theta\coloneqq \mr{Th}(-\theta^*\gamma_{2n})$ the Thom spectrum of the inverse of the pullback of the universal bundle along $\theta \colon B'\ra B\mr{O}(2n)$. This spectrum admits an action by the group-like topological monoid $\hAut(u)$ of self-weak equivalences of $B'$ that commute with $B'\ra B$. A parametrised form of the Pontryagin--Thom construction gives rise to a canonical homotopy class of maps
\begin{equation}\label{equ:param-PT}
\BDiff^\lambda(M)_{\ell_M}\lra \Omega^\infty\mr{MT}\theta\sslash \hAut(u)
\end{equation}
whose effect on homology is subject of the work of Galatius--Randal-Williams mentioned above. 

Their main result implies that if $M$ is simply-connected and of dimension $2n\ge6$, the map \eqref{equ:param-PT}, when regarded as a map onto the path component it hits, induces an isomorphism in homology in a range of degrees depending on the \emph{genus} $g(M,\ell_M)$ of $(M,\ell_M)$. The genus is defined as the maximal number of disjoint embeddings $e\colon S^n\times S^n\backslash \mr{int}(D^{2n})\hookrightarrow M$ such that $(\ell_M\circ e)\colon S^n\times S^n\backslash \mr{int}(D^{2n})\ra B$ is null-homotopic. Their work was extended by Friedrich \cite{Friedrich} to manifolds with certain nontrivial fundamental groups. The version of this result we shall use reads as follows (see \cite[Theorem 12.4.5, Section 12.4.7]{GRWSurvey}).

\begin{theorem}[Friedrich, Galatius--Randal-Williams]\label{thm:stable-homology}
For a polycyclic-by-finite group $\pi$, there exists a function $\varphi_\pi \colon \bN_0 \to \bN_0$ with $\lim_{g\to\infty}\varphi_\pi(g)=\infty$ such that for any closed connected manifold $M$ of dimension $d=2n \geq 6$ with $\pi_1(M)\cong \pi$ and a factorisation as in \eqref{equ:tangential-structure}, the map
\[\BDiff^\lambda(M)_{\ell_M}\lra \Omega^\infty\mr{MT}\theta\sslash \hAut(u),
\]
regarded as a map onto the path-components hit, induces an isomorphism in integral homology
in degrees $\leq \varphi_\pi(g(M,\ell_M))$.\end{theorem}

\begin{remark}
The works mentioned above provide an explicit choice of $\varphi_\pi$, namely $\varphi_\pi(g)=\tfrac{1}{2}(g-h(\pi)-6)$ where $h(\pi)$ is the \emph{Hirsch length} of $\pi$---the number of infinite cyclic summands in the factors of a subnormal series. This choice plays no role in our arguments.
\end{remark}

\cref{thm:stable-homology} serves us to prove a finiteness result for the homology of $\BDiff^\lambda(M)_{\ell_M}$ in a range and under some conditions on $B$. One of the conditions we opt for is technical and certainly not optimal, but it suffices for our purposes. It involves
\begin{enumerate}
\item the $k$-truncation $\tau_{\le k}X$ of a space together with its canonical map $X\ra \tau_{\le k}X$,
\item the \emph{rationalisation above the fundamental group} of a space $X$, i.e.\,the fibrewise rationalisation $X_\bQ^{\mr{fib}}$ of the $1$-truncation $X\ra \tau_{\le 1}X$ which comes with a natural map $X\ra X_\bQ^{\mr{fib}}$ over $\tau_{\le 1}X$ (see the proof of \cref{prop:triantafillou-haut}),
\item the \emph{generalised Eilenberg--MacLane space} $K(G,A)$ associated to a group $G$ with a degree-preserving action on a graded abelian group $A=\bigoplus_{n\ge2}A_n$. This is the based space given as the homotopy orbits $K(A)_{hG}$ of the $G$-action on the Eilenberg--MacLane space $K(A)=\prod_{i\ge2}K(A,n)$. It has the property that there exist isomorphisms
\[
\pi_k(K(G,A))\cong\begin{cases}
G&\text{if }k=1,\\
A_n&\text{if }k=n,\\
0&\text{otherwise,}
\end{cases}
\]
with respect to which the action of $\pi_1(K(G,A))$ on $\pi_n(K(G,A))$ is the given one. There is a preferred twisted cohomology class $\iota \in H^n(K(G,A);A_n)$ with the property that for a connected based space $X$ there is a natural bijection
\[\textstyle{[X,K(G,A)]_*\cong \bigsqcup_{\phi\in\mr{Hom}(\pi_1(X),G)}\prod_{n\ge 2}H^n(X;\phi^*A_n),}\]
induced by pulling back $\iota$.
\end{enumerate}

\begin{theorem}\label{thm:stable-finiteness}
Fix a closed connected manifold $M$ of dimension $d=2n \geq 6$ and a factorisation as in \eqref{equ:tangential-structure} such that
\begin{enumerate}
\item $\pi_1(M)$ and $\pi_1(B)$ are finite,
\item $\pi_k(B)$ is finitely generated for $k\ge2$,
\item $\tau_{\le 2n}B^{\mr{fib}}_{\bQ}$ is weakly equivalent to a generalised Eilenberg--MacLane space.
\end{enumerate}
Then the group $H_k(\BDiff^\lambda(M)_{\ell_M};\bZ)$ is finitely generated for $k \le \varphi_{\pi_1(M)}(g(M,\ell_M))$ where $\varphi_{\pi_1(M)}$ is a function as in \cref{thm:stable-homology}.
\end{theorem}

\begin{example}\label{ex:BO(2n)-example}Taking $\lambda$ to be the canonical factorisation
\[M\xlra{\ell_M}B\mr{O}(2n)\xlra{\mr{id}} B\mr{O}(2n),\]
we have $\BDiff^\lambda(M)_{\ell_M}\simeq \BDiff(M)$. In this case, the second and third assumptions of \cref{thm:stable-finiteness} are always satisfied, since there is a weak homotopy equivalence
\[\textstyle{B\mr{O}(2n)^{\mr{fib}}_{\bQ}\rightarrow  K\big(\bZ/2,\bQ^-[2n]\oplus \bigoplus_{1\le i\le n-1}\bQ^+[4i]\big)}\]
induced by the twisted Euler class and the Pontryagin classes. Here the superscript $\pm$ indicates whether $\bZ/2$ acts trivially or by multiplication with $-1$.
\end{example}

\begin{proof}[Proof of \cref{thm:stable-finiteness} ]
In view of \cref{thm:stable-homology} and the fact that finite groups have Hirsch length 0, it suffices to show that the homology of the path component 
\[(\Omega^\infty\mr{MT}\theta\sslash \hAut(u))_{\ell_M}\subset \Omega^\infty\mr{MT}\theta\sslash \hAut(u)\]
hit by \eqref{equ:param-PT} is degreewise finitely generated. As explained in \cite[Section 12.4.3]{GRWSurvey}, the orbit-stabiliser theorem implies that this path component is equivalent to the homotopy quotient $\Omega^\infty_{\rho_M}\mr{MT}\theta\sslash \hAut(u)_{\rho_M}$ of the path component $\Omega^\infty_{\rho_M}\mr{MT}\theta\subset \Omega^\infty\mr{MT}\theta$ induced by the manifolds $M$ together with the map $\rho_M\colon M\ra B'$ from \eqref{equ:factorisation-tangential-structure}, acted upon by the stabiliser $\hAut(u)_{\rho_M}\subset \hAut(u)$ of $[M,\rho_M]\in \pi_0(\mr{MT}\theta)$. This quotient fits into a fibration
\[\Omega^\infty_{\rho_M}\mr{MT}\theta\lra \Omega^\infty_{\rho_M}\mr{MT}\theta\sslash \hAut(u)_{\rho_M}\lra B\hAut(u)_{\rho_M},
\]
so the Serre spectral sequence shows that it suffices to prove that the fibre has degreewise finitely generated homology groups and that the base is of finite type. To see the former, it suffices to show that the homology of the Thom spectrum $\mr{MT}\theta$ is bounded below and degreewise finitely generated, for which we use the Thom isomorphism $H_k(\mr{MT}\theta;\bZ)\cong H_{k+2n}(B';\bZ_{\omega})$  involving the local system $\bZ_\omega$ induced by the vector bundle $(\lambda\circ u)^*\gamma_{2n}$ over $B'$. To see that the homology $H_{*}(B';\bZ_{\omega})$ is degreewise finitely generated, note that as $\rho_M\colon M\ra B'$ is $n$-connected, the space $B'$ is connected, has finite fundamental group, and its higher homotopy groups are finitely generated up to degree $(n-1)$ as a consequence of \cref{coro:pi1-finite-pik-fg}. In degrees $\ge n$, the group $\pi_k(B')$ agrees (up to passing to a subgroup if $k=n$) with $\pi_k(B)$ which is finitely generated by assumption. \cref{lem:fg-homology-implies-fg-homotopy-finite} thus shows that  $H_{*}(B';\bZ_{\omega})$ is indeed degreewise finitely generated. 

This leaves us with showing that $B\hAut(u)_{\rho_M}$ is of finite type for which we consider the canonical composition
\[
\pi_0(\mr{MT}\theta)\lra \pi_0(\mr{MT}\theta)\otimes\bQ\cong H_0(\mr{MT}\theta;\bQ)\cong H_{2n}(B';\bQ_\omega)
\]
where the last two isomorphisms are induced by the Hurewicz and Thom isomorphisms respectively. We observed above that $\pi_0(\mr{MT}\theta)$ is finitely generated, so the first map has finite kernel, which implies that the stabiliser $\pi_0(\hAut(u))_{\rho_M}$ of $[M,\rho_M]\in \pi_0(\mr{MT}\theta)$ has finite index in the stabiliser $\pi_0(\hAut(u))_{\rho_{M,\bQ}}$ of the image $[M,\rho_M]_\bQ\in H_{2n}(B';\bQ_\omega)$ under the above composition. In view of \cref{prop:finite-type-homotopy-groups} we may instead show that the group of path components of $\hAut(u)_{\rho_{M,\bQ}}$ is of finite type and that its higher homotopy groups are finitely generated. As $u$ is $n$-co-connected, it is in particular $2n$-co-connected, so it follows from \cref{lem:truncation-relative-mapping-space} below that we may replace $\hAut(u)_{\rho_{M,\bQ}}$ by the stabiliser
 \[
 \hAut(\tau_{\le 2n}u)_{\rho_{M,\bQ}}\subset \hAut(\tau_{\le 2n}u)
 \]
 of $[M,\rho_M]_\bQ\in H_{2n}(B';\bQ_\omega)\cong H_{2n}(\tau_{\le 2n}B';\bQ_\omega)$. We now consider the fibration sequence
\begin{equation}\label{equ:truncated-fibre-sequence}
\hAut(\tau_{\le 2n}u)\lra \hAut(\tau_{\le 2n}B')\xrightarrow{(\tau_{\le 2n}u)\circ(-)} \Map(\tau_{\le 2n}B',\tau_{\le 2n}B)
\end{equation}
with homotopy fibre taken over $\tau_{\le 2n}u$. The base and total spaces of this fibration have at all basepoints finitely generated higher homotopy groups and polycyclic-by-finite fundamental group by \cref{cor:section-spaces-F-coconnected} since $\tau_{\le 2n}B'$ as finite type as a result of \cref{cor:fg-homotopy-implies-fg-homology}, so the same holds for the fibre. Thus it only remains to show that the group $\pi_0(\hAut(\tau_{\le 2n}u)_{\rho_{M,\bQ}})$ is of finite type. As the groups $\pi_1(B)$ and $\pi_1(B')$ are finite, it is straightforward to see that this group is commensurable up to finite kernel to the analogous group $\pi_0(\hAut^{\pi_1}_*(\tau_{\le 2n}u)_{\rho_{M,\bQ}})$ of pointed homotopy automorphisms over $\tau_{\le 2n}u$ (considered as a pointed map) that induce the identity on fundamental group. Being of finite type is invariant under commensurability up to finite kernel by \cref{lem:f-infty-commensurable}, so we may restrict our attention to $\pi_0(\hAut^{\pi_1}_*(\tau_{\le 2n}u)_{\rho_{M,\bQ}})$. Using the pointed analogue of \eqref{equ:truncated-fibre-sequence}, we see that this group fits into an extension whose kernel is a quotient of the polycyclic-by-finite group $\pi_1(\Map_*(\tau_{\le 2n}B,\tau_{\le 2n}B'),\tau_{\le 2n}u)$ and whose quotient is the stabiliser
\begin{equation}\label{equ:stabiliser-homology-class-and-map}
\pi_0(\hAut^{\pi_1}_*(\tau_{\le 2n}B')_{\rho_{M,\bQ}})_{\tau_{\le 2n}u}\le \pi_0(\hAut^{\pi_1}_*(\tau_{\le 2n}B')_{\rho_{M,\bQ}})\end{equation} of $[\tau_{\le 2n}u]\in \pi_0(\Map_*(\tau_{\le 2n}B',\tau_{\le 2n}B))$ with respect to the action by precomposition. Being of finite type is preserved under extensions (see \cref{lem:f-infty}) and polycyclic-by-finite groups are of finite type (see \cref{lem:group-classes-closure}), so the claim follows once we show that the subgroup \eqref{equ:stabiliser-homology-class-and-map} is of finite type. Since $B$ has finitely generated homotopy groups, an induction over a Postnikov tower shows that the map
\[\pi_0(\Map_*(\tau_{\le 2n}B',\tau_{\le 2n}B))\lra\pi_0(\Map_*(\tau_{\le 2n}B',\tau_{\le 2n}B^{\mr{fib}}_\bQ))\]
induced by postcomposition with the fibrewise rationalisation has finite kernel, so the stabiliser of $[\tau_{\le 2n}u]\in \pi_0(\Map_*(\tau_{\le 2n}B',\tau_{\le 2n}B))$ has finite index in the stabiliser of the image in $\pi_0(\Map_*(\tau_{\le 2n}B',\tau_{\le 2n}B^{\mr{fib}}_\bQ))$. As $\tau_{\le 2n}B^{\mr{fib}}_\bQ$ is weakly equivalent to a generalised Eilenberg-MacLane spaces, we see that the subgroup \[\pi_0(\hAut^{\pi_1}_*(\tau_{\le 2n}B')_{\rho_{M,\bQ}})_{\tau_{\le 2n}u}\le  \pi_0(\hAut^{\pi_1}_*(\tau_{\le 2n}B'))\] has finite index in the subgroup given as the common stabiliser of $[M,\rho_M]_\bQ\in H_{2n}(B';\bQ_\omega)$ and the classes $\sigma_k\in H^k(\tau_{\le 2n}B';\pi_k(B)\otimes \bQ)$ for $2\le k\le 2n$ induced by $\tau_{\le 2n}u$, where $\pi_1(\tau_{\le 2n}B')$ acts on $\pi_k(B)$ via $\pi_1(\tau_{\le 2n}u)$. This exhibits the claim as a consequence of \cref{prop:triantafillou-haut}.
\end{proof}

We finish this section with the lemma promised in the previous proof. Given CW complexes $X$ and $X'$ and maps $p\colon X\ra B$ and $p'\colon X'\ra B$, we denote by $\mr{Map}^B(X,X')$ the space of pairs of a map $f\colon X\ra X'$ and a homotopy from $p$ to $p'\circ f$, in the compact-open topology.
\begin{lemma}\label{lem:truncation-relative-mapping-space}
If $p\colon X\ra B$ is $n$-co-connected, then the map induced by $n$-truncation
\[\Map^B(X,X)\lra \Map^{\tau_{\le n}B}(\tau_{\le n}X,\tau_{\le n}X)\]
is a weak homotopy equivalence.
\end{lemma}

\begin{proof}
Denoting the $n$-truncation by $\tau\colon X\ra \tau_{\le n}X$, we consider the commutative diagram
\begin{center}
\begin{tikzcd}[column sep=0.2cm]
\Map^B(X,X)\arrow[rr]\arrow[dr,"\tau\circ(-)",swap]&&\Map^{\tau_{\le n}B}(\tau_{\le n}X,\tau_{\le n}X)\arrow[dl,"(-)\circ\tau"]\\[-4pt]
&\Map^{\tau_{\le n}B}(X,\tau_{\le n}X)&.
\end{tikzcd}
\end{center}
The claim will follow by showing that the diagonal maps are weak equivalences. The right diagonal map agrees with the induced map on horizontal homotopy fibres of the square
 \begin{center}
\begin{tikzcd}[column sep=1.5cm]
\Map(\tau_{\le n}X,\tau_{\le n}X)\rar{(\tau_{\le n}p) \circ(-)}\arrow[d,swap,"(-)\circ\tau"]&\Map(\tau_{\le n}X,\tau_{\le n}B)\dar{(-)\circ\tau}\\[-4pt]
\Map(X,\tau_{\le n}X)\rar{(\tau_{\le n}p) \circ(-)}&\Map(X,\tau_{\le n}B)
\end{tikzcd}
\end{center}
whose vertical maps are weak equivalences by obstruction theory, so the right diagonal map is a weak equivalence as well. That the left diagonal arrow is a weak equivalence follows from the fact that the square
\begin{center}
\begin{tikzcd}
X\rar\arrow[d,"p",swap]&\tau_{\le n}X\dar{\tau_{\le n}p}\dar\\[-4pt]
B \rar&\tau_{\le n}B
\end{tikzcd}
\end{center}
is homotopy cartesian because $p$ is $n$-co-connected.
\end{proof}

\section{Finiteness properties of embedding spaces}\label{sec:emb-calc}
Our proof of \cref{thm:fg pi} involves various comparisons between spaces of diffeomorphisms and spaces of self-embeddings. In this section we focus on proving finiteness properties of spaces of embeddings by use of  \emph{embedding calculus}, as developed by Goodwillie, Klein, and Weiss \cite{Weiss,WeissErratum,GoodwillieWeiss,GoodwillieKlein}. These results might be useful in other situations, so we phrase them in more generality than needed for the proof of \cref{thm:fg pi}.

\subsection{Triad embeddings}\label{section:triads} We call a submanifold $N\subset M$ of a manifold $M$ a \emph{compact triad-pair} if $N$ and $M$ are compact and $N$ comes with a decomposition $\partial N= \partial_0N \cup \partial_1N$ of its boundary into (possibly empty or disconnected) compact submanifolds $\partial_0 N$ and $\partial_1 N$ that meet in a corner such that $\partial_0N\times [0,1)=(\partial M\times [0,1))\cap N$ for a choice of collar $\partial M\times [0,1]\subset M$. Given a compact triad-pair $N\subset M$, we denote by $\Emb_{\partial_0}(N,M)$ the space of embeddings $e\colon N\hookrightarrow M$ that agree with the inclusion $N\subset M$ on a neighborhood of $\partial_0N$, in the smooth topology. 

\begin{remark}Embedding calculus as recalled below is usually stated for triads with $\partial_1N=\varnothing$, but the general case can easily be deduced from this, as explained in \cite[Remark 1.2]{KKSurfaces}.
\end{remark}

\subsection{Recollection of embedding calculus}\label{sec:emb-calc-recollection}

Let $N \subset M$ be a compact triad-pair with $\dim(M) = d \geq 3$. Weiss' \emph{Embedding calculus} \cite{Weiss} provides a tower of approximations
\[\cdots \lra T_3 \Emb_{\partial_0}(N,M) \lra T_2 \Emb_{\partial_0}(N,M) \lra T_1 \Emb_{\partial_0}(N,M)\]
under $\Emb_{\partial_0}(N,M)$, such that the map
\begin{equation}\label{equ:stages}\Emb_{\partial_0}(N,M) \lra T_r \Emb_{\partial_0}(N,M)\end{equation}
is $(-(h-1)+r(d-2-h))$-connected where $h$ is the \emph{relative handle dimension} of $(N,\partial_0N)$, i.e.\,the minimum over the maximal indices of handle decompositions of $N$ relative to $\partial_0N$ (see \cite[Corollary 2.5]{GoodwillieWeiss}). The space $T_1 \Emb_{\partial_0}(N,M)$ is weakly homotopy equivalent to the space $\mr{Bun}_{\partial_0}(TN,TM)$ of bundle monomorphisms that extend the identity on $TN|_{\partial_0 N}$, and under this identification \eqref{equ:stages} can be identified with the derivative (cf.\,\cite[Example 10.3]{Weiss}).

Writing $\iota \colon N \hookrightarrow M$ for the inclusion, the homotopy fibres (also called the \emph{layers})
\begin{equation}\label{equ:layers}\hofib_{T_{r-1}(\iota)}(T_r \Emb(N,M) \lra T_{r-1} \Emb(N,M))\end{equation}
admit an explicit description in terms of configuration spaces for $r\ge2$. Details are given in \cite[Section 3.3.2]{KRWAlgebraic}---following \cite{Weiss,WeissErratum}---in the special case of manifolds with boundary and embedding of codimension zero, but this can be easily generalised. It suffices for us to know that this homotopy fibre is given by a relative section space 
\[\mr{Sect}_{\partial_0 C_r[N])}(\iota^* Z_r(M) \to C_r[N])\]
of a fibration over a pair $(C_r[N],\partial_0 C_r[N])$ 
that is homotopy equivalent to a finite CW pair.
Roughly, $C_r[N]$ is obtained from the configuration space of $r$ unordered points in $N$ by allowing particles to be infinitesimally close, and $\partial_0 C_r[N]$ consists of configurations where either at least two points are infinitesimally close or at least one lies in $\partial_0 N$. The fibre of this fibration agrees with the total homotopy fibre of the cubical diagram (see \cite[Def.\ 5.5.1]{MunsonVolic} for a definition)
\begin{equation}\label{eqn:tohofib} \underline{r} \coloneqq \{1,\ldots,r\} \supset I \longmapsto \Emb(\underline{r} \setminus I,M).\end{equation}

\subsection{Finiteness properties through embedding calculus}Assuming the handle dimension of $(N,\partial_0N)$ is so that the connectivity range of \eqref{equ:stages} increases with $r$, we may prove finiteness properties of $\Emb_{\partial_0}(N,M)$ by separately considering $T_1\Emb_{\partial_0}(N,M)$ and all layers \eqref{equ:layers}. This is carried out in \cref{prop:induction-tower} below, after two preliminary lemmas.

\begin{lemma}\label{lem:bundle-maps} Let $N\subset M$ be a compact triad-pair such that at all basepoints $\pi_k(M)$ is finitely generated for $k \geq 2$ and polycyclic-by-finite for $k=1$. Then at all basepoints $\pi_k(\mr{Bun}_{\partial_0}(TN,TM))$ is finitely generated for $k \geq 2$ and polycyclic-by-finite for $k = 1$.
\end{lemma}

\begin{proof}There is a fibre sequence
	\[\mr{Sect}_{\partial_0}(\mr{Lin}(TN,f^*TM) \to N) \lra \mr{Bun}_{\partial_0}(TN,TM) \lra \mr{Map}_{\partial_0}(N,M)\]
	with fibre taken over a map $f\in  \mr{Map}_{\partial_0}(N,M)$. Here $\mr{Lin}(TN,f^*TM) \to N$ is the bundle over $N$ with fibre over $n \in N$ given by the space of linear injections $T_n N \to T_{f(n)}M$, which is homotopy equivalent to $\mr{GL}_d(\bR)/\mr{GL}_{k}(\bR)$ where $k=\mr{codim}(N\subset M)$. The derivative of the inclusion $\partial_0N\subset M$ induces a section over $\partial_0N$. \cref{lem:section-spaces-finite-CW} now shows that base and fibre of this fibration have polycyclic-by-finite fundamental groups and finitely generated higher homotopy groups, so the same holds for the total space, as claimed.
\end{proof}

\begin{lemma}\label{lem:tohofib} Let $M$ be a compact smooth manifold of dimension $d \geq 3$ with $\pi_1(M)$ finite at all basepoints. Then the $k$th homotopy group of the total homotopy fibre of \eqref{eqn:tohofib} is finitely generated for $k \geq 2$ and finite for $k = 1$.\end{lemma}

\begin{proof}Since total homotopy fibres can be computed iteratively \cite[Proposition 5.5.4]{MunsonVolic}, it suffices to prove that for all $s \leq r$, $\pi_k(\mr{Emb}(\ul{s},M))$ is finitely generated for $k \geq 2$ and finite for $k=1$. This will follow by induction over $s$ from the fibre sequence 
\[M \setminus \{s-1 \text{ points}\} \lra \Emb(\underline{s},M) \lra \Emb(\underline{s-1},M)\]
once we prove that $\pi_k(M \setminus \{s-1 \text{ points}\})$ is finitely generated for $k \geq 2$ and finite for $k=1$. This follows from Lemma \ref{lem:fg-homology-implies-fg-homotopy-finite}, using the fact that the manifold $M \setminus \{s-1 \text{ points}\}$ has the same fundamental group as $M$ (using transversality and $d \geq 3$) and that its homology is finitely generated in each degree (using the Mayer--Vietoris sequence).
\end{proof}

\begin{remark}For a connected $d$-manifold $M$ with $\pi_{d-1}(M)$ and $\pi_d(M)$ finitely generated and $d\ge3$,  $\pi_{d-1}(\Emb(\underline{2},M))$ is finitely generated \emph{if and only if} $\pi_1(M)$ is finite \cite[Thm 2]{Hansen}.\end{remark}

\begin{proposition}\label{prop:induction-tower}
		Let $N\subset M$ be a compact triad-pair. If $\partial_0 N\subset N$ has relative handle dimension at most $d-3$ and at all basepoints $\pi_1(M)$ is finite, then at all basepoints $\pi_k(\Emb_{\partial_0}(N,M))$ is finitely generated for $k \geq 2$ and polycyclic-by-finite for $k=1$.
\end{proposition}

\begin{proof}We fix an embedding $\iota\in \Emb_{\partial_0}(N,M)$. By the assumption on the handle dimension of $N$, the connectivity of the map $T_r\colon \Emb_{\partial_0}(N,M) \to T_r \Emb_{\partial_0}(N,M)$ tends to $\infty$ with $r$. Hence it suffices to show that $\pi_k(T_r \Emb_{\partial_0}(N,M);T_r(\iota))$ has the desired property for all $r\ge1$, which we do by induction over $r\ge1$. The base case is \cref{lem:bundle-maps} since $\pi_k(M)$ is finitely generated by \cref{coro:pi1-finite-pik-fg}. For the induction step we consider the fibration sequence
	\[\mr{Sect}_{\partial_0}\big(\iota^* Z_r(M) \to C_r[N]\big) \lra T_r \Emb_{\partial_0}(N,M) \lra T_{r-1}\Emb_{\partial_0}(N,M)\]
	with fibre taken over $T_{r-1}(\iota)$. By the induction hypothesis, the homotopy groups of the base at the basepoint $T_{r-1}(\iota)$ are finitely generated for $k\ge2$ and polycyclic-by-finite for $k=1$, so it suffices to show that the homotopy groups of the path components of the fibre satisfy the same property. As mentioned above, the pair $(C_r[N],\partial_0 C_r[N])$ is homotopy equivalent to a finite CW pair. By \cref{lem:tohofib} the homotopy groups the fibre of $\iota^* Z_r(M) \to C_r[N]$ have the desired property, so it follows from \cref{lem:section-spaces-finite-CW} that the homotopy groups of all components of the section space have this property and hence the same holds for $\pi_k(T_r \Emb_{\partial_0}(N,M);T_r(\iota))$.
\end{proof}

\begin{remark}\cref{prop:induction-tower} was certainly known to experts in embedding calculus. A variant of this result for $1$-connected manifolds was stated by Goodwillie in \cite{GoodwillieMathoverflow} and a variant for certain self-embedding spaces featured in \cite[Proposition 3.15]{Kupers}.\end{remark}

\section{Finiteness properties of diffeomorphism groups}\label{sec:proof}
The proof of \cref{thm:fg pi} is divided in two steps: in \cref{sec:pi2 finite} we prove the claim under the additional assumption that $M$ is connected and has finite second homotopy group, from which we deduce the general case in \cref{sec:proof-of-main-thm} by surgery. The first step involves a variant of a fibre sequence known as the \emph{Weiss fibre sequence}, and we explain this first.

\subsection{The Weiss fibre sequence}\label{sec:weiss}
Let $M$ be a compact smooth manifold with and $N\subset\partial M$ a compact codimension $0$ submanifold of the boundary. There is a fibre sequence of the form
\begin{equation}\label{equ:weiss-fibration-sequence}
B\Diff_{\partial}(M) \lra B\Emb^{\cong}_{\partial M\setminus\mathrm{int}(N)}(M) \lra B^2 \Diff_\partial(N\times [0,1])
\end{equation}
where $\Diff_\partial(M)$ is the group of diffeomorphisms of $M$ in the smooth topology that fix a neighborhood of the boundary pointwise and $\smash{\Emb^{\cong}_{\partial M\setminus\mathrm{int}(N)}}(M)$ is the topological monoid of self-embeddings of $M$ in the smooth topology that agree with the identity on a neighborhood of $\partial M\setminus\mathrm{int}(N)$ and are isotopic through such embeddings to a diffeomorphism of $M$ fixing $\partial M$ pointwise. The space $B^2 \Diff_\partial(N\times [0,1])$ is the delooping of $B\Diff_\partial(N\times [0,1])$ with respect to the $A_\infty$-structure induced by juxtaposition.

A non-delooped form of this fibre sequence follows from the isotopy extension theorem and featured in the special case $N=D^{d-1}$ in Weiss'  work on Pontryagin classes \cite[Remark 2.1.3]{WeissDalian}. Building on Weiss' work, the third-named author \cite{Kupers} proved finiteness results for automorphism spaces of 2-connected manifolds for which he constructed the delooped sequence \eqref{equ:weiss-fibration-sequence} for $N=D^{d-1}$ (see Section 4 loc.cit.). The general case follows in the same way.

\subsection{The case of finite second homotopy group}\label{sec:pi2 finite}
The goal of this section is to prove a weaker version of \cref{thm:fg pi}, which additionally assumes that $M$ is connected and $\pi_2(M)$ is finite. The latter is only used once, to show that a set of embeddings is finite in \cref{lem:path-compo-finite}.

\begin{theorem}\label{thm:pi2 finite}
Let $M$ be a closed connected manifold of dimension $2n\ge6$. If $\pi_1(M)$ and $\pi_2(M)$ are finite, then the groups $\pi_k(\BDiff(M))$ are finitely generated for all $k\geq 2$.
\end{theorem}

Given $M$ as in \cref{thm:pi2 finite}, we fix a handle decomposition of $M$ and decompose $M$ into two codimension $0$ submanifolds
\beq
M = M^{\leq 2}\cup M^{>2}
\eeq
where $M^{\leq 2}$ is the union of the handles of index $\le 2$ and $M^{>2}=M\backslash \mr{int}(M^{\leq 2})$. These submanifolds intersect in $\partial M^{\leq 2} = \partial M^{>2}$. We shall also consider the \emph{stabilised manifolds}
\beq
M^{>2}_g\coloneqq M^{>2}\sharp (S^n\times S^n)^{\sharp g},\quad\text{and}\quad M_g \coloneqq M\sharp (S^n\times S^n)^{\sharp g} \quad\text{for}\quad g\ge0,
\eeq 
where we model the first connected sum by gluing the manifold $([0,1]\times \partial M_g^{>2})\sharp (S^n\times S^n)^{\sharp g}$ to $M^{>2}$ along $\partial M^{>2}\subset M^{>2}$, so we have a canonical identification  $\partial M^{>2}_g\cong\partial M^{>2}$. The second connected sum is obtained from $M_g^{>2}$ by gluing on $M^{\le 2}$ along $\partial M^{>2}=\partial M^{\le 2}$. Fixing a choice of closed embedded disc $D^{2n}\subset \mr{int}(M^{\le 2})$, we moreover define
\[
M_{g,1} \coloneqq M_g\backslash \mr{int}(D^{2n}).
\]
so we have inclusions
\[M^{>2}\subset M_g^{>2}\subset M_{g,1}\subset M_g\quad\text{and}\quad M^{\le 2}\subset M_g.\]
Note that these definitions include the case $g=0$ where we have $M_0=M$, and $M^{>2}_0=M^{>2}$.

Our proof of \cref{thm:pi2 finite} centres around two fibre sequences
\begin{equation}\label{eq:weiss-K_g}
B\Diff_{\partial}(M_g^{>2})\lra
B\Emb^{\cong}(M_g^{>2})\lra B^2\Diff_{\partial}(\partial M^{>2}\times I)\quad\text{and}
\end{equation}
\begin{equation}\label{eq:restriction 2-sk}
\Emb(M^{\leq 2},M_g)_{\iota_2}\lra \BDiff_{\partial}(M_g^{>2})\lra \BDiff(M_g)_{\iota_2}.
\end{equation} 
The sequence \eqref{eq:weiss-K_g} is the Weiss fibre sequence \eqref{equ:weiss-fibration-sequence} for $M=M_g^{>2}$ and $N = \partial M^{>2}$. To explain the second sequence, we write $\iota_2\colon M^{\le2}\hookrightarrow M_g$ for the inclusion and denote by $\Emb(M^{\leq 2},M_g)_{\iota_2}\subset \Emb(M^{\leq 2},M_g)$ the component of $\iota_2$ and by $\Diff(M_g)_{\iota_2}\subset \Diff(M_g)$ the subgroup of the path components that stabilise $[\iota_2]\in  \pi_0(\Emb(M^{\leq 2},M_g))$ under postcomposition. With this notation in place, \eqref{eq:restriction 2-sk} is induced by extending diffeomorphisms by the identity along $M_g^{>2}\subset M_g$.

\vspace{.5em}

Using these fibre sequences, we will prove \cref{thm:pi2 finite} via an induction involving the following four statements. During the proof, we abbreviate \emph{finitely generated} by \emph{f.g.}\ and fix a manifold $M$ as in \cref{thm:pi2 finite} as well as a function $\varphi_{\pi_1(M)}\colon \bN_0\ra\bN_0$ as in \cref{thm:stable-homology}.
\begin{enumerate}
		\item[(a)$_k$] $\pi_i(B^2\Diff_\partial(\partial M^{>2} \times I))$ is f.g.\ for all $i$ with $2 \leq i \leq k+1$,
		\item[(b)$_{k}$] $\pi_i(B\Diff(M_g)_{\iota_2})$ is f.g.\ for all pairs $(i,g)$ with $1\leq i \leq k$ and $g\ge0$,
		\item[(c)$_{k}$]$H_i(B\Diff(M_g)_{\iota_2};\bZ)$ is f.g.\ for all pairs $(i,g)$ with $1\leq i \leq k+1\le\varphi_{\pi_1(M)}(g)$,
		\item[(d)$_{k}$] $H_i(B\Diff_\partial(M^{>2}_g);\bZ)$ is f.g.\ for all pairs $(i,g)$ with $1\le i \leq k\le\varphi_{\pi_1(M)}(g)$.
\end{enumerate}

\noindent The induction steps rely on the following four assertions which we justify later.
\begin{enumerate}[label=\textbf{(A\,\arabic*)}, ref={(A\,\arabic*)} ]
\item\label{A-emb-homotopy}At all basepoints, the higher homotopy groups of $\Emb(M^{\leq 2},M)$ and $\Emb^{\cong}(M_g^{>2})$ for $g \geq 0$ are f.g.\ and their fundamental groups are polycyclic-by-finite.
\item\label{A-emb-homology}The spaces $\Emb(M^{\leq 2},M)_{\iota_2}$ and $B\Emb^{\cong}(M_g^{>2})$ are of finite type for $g\ge0$. 
\item\label{A-mcg2} The group $\pi_0(\Diff(M_g)_{\iota_2})$ is of finite type for $g\ge0$.
\item\label{A-diff-homology} (b)$_{k}$ implies (c)$_{k}$.
\end{enumerate}

\begin{proof}[Proof of \cref{thm:pi2 finite} assuming \textup{\ref{A-emb-homotopy}--\ref{A-diff-homology}}]Since  \cref{thm:pi2 finite} is a statement about higher homotopy groups, we may replace $\BDiff(M)$ by the space $\BDiff(M)_{\iota_2}$. Combining this with \ref{A-emb-homotopy}, the fibre sequence \eqref{eq:restriction 2-sk} for $g=0$, and the fact that subgroups of polycyclic-by-finite groups are f.g., we see that it suffices to show that $\pi_k(\BDiff_\partial(M^{>2}))$ is f.g.\ for $k\ge2$. Combining \eqref{eq:weiss-K_g} for $g=0$ with \ref{A-emb-homotopy}, this in turn reduces it to showing that $\pi_k(B^2\Diff_{\partial}(\partial M^{>2}\times I))$ is f.g.\ for $k\ge2$; this is statement (a)$_k$ for all $k\ge0$. 

The proof now proceeds by simultaneously proving (a)$_{k}$--(d)$_{k}$ for all $k\ge0$ via an induction on $k$. The case $k=0$ holds either trivially or as a result of \ref{A-mcg2}, so the induction is completed by the following chain of implications:

\medskip

	\noindent {(b)$_{k}$ and (c)$_{k}$ $\Rightarrow$ (d)$_{k+1}$.} Consider the Serre spectral sequence of \eqref{eq:restriction 2-sk}, which has the form
	\[{}^\mr{I}E^2_{p,q} = H_p\big(B\Diff(M_g)_{\iota_2};H_q(\Emb(M^{\leq 2},M_g)_{\iota_2};\bZ)\big) \Longrightarrow H_{p+q}(B\Diff_\partial(M^{>2}_g);\bZ).\]
	By \ref{A-emb-homology}, the $\bZ[\pi_0(\Diff(M_g)_{\iota_2}]$-module $H_q(\Emb(M^{\leq 2},M_g)_{{\iota_2}};\bZ)$ is f.g.\ as an abelian group for all $q\ge0$. Therefore, using (b)$_{k}$ and \ref{A-mcg2}, an application of \cref{cor:fg-homotopy-implies-fg-homology} shows that ${}^\mr{I}E^2_{p,q}$ is f.g.\ for $p \leq k$, while (c)$_{k}$ implies that ${}^\mr{I} E^2_{k+1,0}$ is f.g.\ for $k+1\le \varphi_{\pi_1(M)}(g)$. We thus have that ${}^\mr{I} E^2_{p,q}$ is f.g.\ for $p+q\le k+1$ and $k+1\le \varphi_{\pi_1(M)}(g)$, so $H_{p+q}(B\Diff_\partial(M^{>2}_g);\bZ)$ is f.g.\ in the same range; this is (d)$_{k+1}$.
	
\medskip

	\noindent {(a)$_{k}$ and (d)$_{k+1}$ $\Rightarrow$ (a)$_{k+1}$.} 
	Pick $g\ge0$ with $k+1\le \varphi_{\pi_1(M)}(g)$ and consider the Serre spectral sequence induced by the fibre sequence \eqref{eq:weiss-K_g}
	\[{}^{\mr{II}}E^2_{p,q} = H_p\big(B^2\Diff_\partial(\partial M^{>2} \times I);H_q(B\Diff_\partial(M^{>2}_g;\bZ)\big) \Longrightarrow H_{p+q}(B\Emb^{\cong}(M^{>2}_g);\bZ).\]		
	From (a)$_k$ and \cref{cor:fg-homotopy-implies-fg-homology}, it follows that ${}^{\mr{II}}E^2_{p,0}$ is f.g.\ for $p\le k+1$. From (d)$_{k+1}$, we obtain that ${}^{\mr{II}}E^2_{0,q}$ is f.g.\ for $q\le k+1$.  As $B^2\Diff_\partial(\partial M^{>2} \times I)$ is 1-connected, the local systems involved in the $E^2$-page are trivial, so the universal coefficient theorem implies that ${}^{\mr{II}}E^2_{p,q}$ is f.g.\ for $p,q\le k+1$. Since $E^\infty_{p,q}$ is f.g.\ as a result of \ref{A-emb-homology}, we conclude that ${}^{\mr{II}}E^2_{p,0} = H_{p}(B^2\Diff_\partial(\partial M^{>2} \times I);\bZ)$ is f.g.\ for $p\le k+2$, so (a)$_{k+1}$ follows from \cref{lem:fg-homology-implies-fg-homotopy-finite}.
	
		\medskip
	
		\noindent {(a)$_{k+1}$ $\Rightarrow$ (b)$_{k+1}$.}
	Using \ref{A-emb-homotopy} and the fact that subgroups of polycyclic-by-finite groups are finitely generated, this follows from the fibre sequences \eqref{eq:weiss-K_g} and \eqref{eq:restriction 2-sk}.
	
	\medskip
	
		\noindent {(b)$_{k+1}$ $\Rightarrow$ (c)$_{k+1}$.} 
	This is \ref{A-diff-homology}.
\end{proof}

In the following four subsections, we establish the remaining assertions \ref{A-emb-homotopy}--\ref{A-diff-homology}.

\subsection*{Assertion \ref{A-emb-homotopy}: \textnormal{\textit{$\Emb(M^{\leq 2},M)$ and $\Emb(M_g^{>2})$ have f.g.\ homotopy groups and polycyclic-by-finite fundamental groups}}} 
We begin with the following auxiliary lemma.

\begin{lemma}\label{lem:handle-bounds}
Fix $g\ge0$ and let $M$ be as in \cref{thm:pi2 finite}. The following manifolds have handle dimension $\le 2n-3$:
\begin{enumerate}
	\item \label{enum:hd-mle2} $M^{\le 2}$ and
	\item \label{enum:hd-mge2} $M_g^{>2}$. 
\end{enumerate}	
Moreover, the following groups are finite:
\begin{enumerate} \setcounter{enumi}{2}
\item \label{enum:pi1-mgge2} $\pi_1(M_g^{>2})$,
\item \label{enum:pi1-mg} $\pi_1(M_g)$, and
\item \label{enum:pi2-mg} $\pi_2(M_g)$.
%\item $\pi_k(M^{\le2})$ and $\pi_k(M_g^{>2})$ for $k=1$, and
%\item $\pi_k(M_g)$ and $\pi_k(M_{g,1})$ for $k\le2$ .
\end{enumerate}
\end{lemma}

\begin{proof}
By construction, the handle dimension of $M^{\le 2}$ is at most $2\le 2n-3$. Reversing the handle decomposition of $M$, we see that the handle dimension of $M^{>2}$ is at most $2n-3$ as well. The manifold $M_g^{>2}$ can obtained from $M^{>2}$ by attaching only $n$-handles, so the handle dimension of $M_g^{>2}$ is at most $2n-3$ too. This completes the proof of \ref{enum:hd-mle2} and \ref{enum:hd-mge2}. 

To prove \ref{enum:pi1-mgge2} note that since $M$ is obtained from $M^{\le2}$ by attaching handles of index at least $3$ and from $M^{>2}$ by attaching handles of index at least $2n-3\ge3$, both inclusions $M^{>2}\subset M$ and $M^{\le 2}\subset M$ are $2$-connected, so $\pi_1(M^{>2})$ and $\pi_1(M^{\le2})$ are finite since $\pi_1(M)$ is finite by assumption. This also implies finiteness of $\pi_1(M_g^{>2})$, since we already noted that $M_g^{>2}$ can obtained from $M^{>2}$ by attaching $n$-handles.

For \ref{enum:pi1-mg} and \ref{enum:pi2-mg}, we use that $M_g$ is obtained from $M_{g,1}$ by attaching a $2n$-handle to conclude that (a) $\pi_k(M_{0,1})$ is finite for $k\le2$ since this holds for $M=M_0$, and that (b) it suffices to show the result for $M_{g,1}$. Now $M_{g,1}$ is obtained from $M_{0,1}$ by attaching $n$-handles and $\pi_k(M_{0,1})$ is finite for $k\le2$, so $\pi_1(M_{g,1})$ is finite for $k\le 2$ and the proof is completed.
\end{proof}

Assertion \ref{A-emb-homotopy} now follows by applying \cref{prop:induction-tower} to $M^{\leq 2} \subset M$ and $M_g^{> 2} \subset M_g^{> 2}$. The assumptions hold by \cref{lem:handle-bounds} \ref{enum:hd-mle2} and \ref{enum:hd-mge2}, together with \ref{enum:pi1-mgge2} and \ref{enum:pi1-mg} for $g=0$.

\subsection*{Assertion \ref{A-emb-homology}: \textnormal{\textit{$\Emb(M^{\leq 2},M)_{\iota_2}$ and $B\Emb^{\cong}(M_g^{>2})$ are of finite type}}}
We first prove:

\begin{lemma}\label{lem:path-compo-finite}
		For $M$ as in \cref{thm:pi2 finite}, the set $\pi_0(\Emb(M^{\leq 2},M_g))$ is finite for $g\ge0$.
	\end{lemma}
	\begin{proof}
		Recall from \cref{sec:emb-calc} that the derivative map
		\begin{equation}\label{equ:derivative-embeddings}
		\Emb(M^{\leq 2},M_g)\lra \mathrm{Bun}(TM^{\leq 2},TM_g)
		\end{equation}
		to the space $\mr{Bun}(TM^{\leq 2},TM_g)$ of bundle monomorphisms from $TM^{\leq 2}$ to $TM_g$ can be identified with the first stage of the embedding calculus tower, so the map is $(-(h-1)+(2n-2-h))$-connected where $h$ is the handle dimension of $M^{\leq 2}$. By construction, we have $h\le 2$, so \eqref{equ:derivative-embeddings} is $(2n-5)$-connected and hence a bijection on components as we assumed $2n\ge6$. It thus suffices to show that $\pi_0(\mathrm{Bun}(TM^{\leq 2},TM_g))$ is finite for which we use the fibration
		\[\mr{Sect}(\mathrm{Iso}(TM^{\leq 2},f^*TM_g) \to M^{\leq 2}) \lra \mr{Bun}(TM^{\leq 2},TM_g)\lra \Map(M^{\leq 2},M_g)\] of the proof of \cref{lem:bundle-maps}, with fibre taken over a map $f\in \Map(M^{\leq 2},M_g)$. The long exact sequence of this fibration reduces the claim further to showing that both $\pi_0(\Map(M^{\leq 2},M_g))$ and $\pi_0(\mr{Sect}(\mathrm{Iso}(TM^{\leq 2},f^*TM_g))\to M^{\leq 2})$ are finite, for any choice of $f\in \Map(M^{\leq 2},M_g)$. Since the first two homotopy groups of $M_g$ are finite by \cref{lem:handle-bounds} \ref{enum:pi1-mg} and \ref{enum:pi2-mg}, and by \cref{lem:handle-bounds} \ref{enum:hd-mle2} the manifold $M^{\le2}$ has handle dimension at most $2$, the set $\pi_0(\Map(M^{\leq 2},M_g))$ is finite by \cref{lem:section-spaces-finite-CW}. The same lemma also takes care of the second set, since the fibre of $\mathrm{Iso}(TM^{\leq 2},f^*TM_g)\to M^{\leq 2}$ is homeomorphic to $\mr{GL}_{2n}(\bR)$, which has two homotopy equivalent path components  whose first two homotopy groups are finite: $\pi_1(\mr{GL}_{2n}(\bR))$ is of order $2$ and $\pi_2(\mr{GL}_{2n}(\bR))$ vanishes.
		\end{proof}

\cref{lem:path-compo-finite} allows us to prove the following.

\begin{lemma}\label{lem:A-mcg}For $g\ge0$, the groups $\pi_0(\Diff_\partial(M_g^{>2}))$ and $\pi_0(\Emb^{\cong}(M_g^{>2}))$ are of finite type. Moreover, $\pi_0(\Diff_\partial(M_g^{>2})$ has polycyclic solvable subgroups in the sense of \cref{def:polyclic-solvable-subgroups}.\end{lemma}
\begin{proof}
We consider the exact sequence 	\[
		\pi_1(\Emb(M^{\leq 2},M_g),\iota_2)\lra\pi_0(\Diff_{\partial}(M^{>2}_g))\lra\pi_0(\Diff(M_g)) \lra\pi_0(\Emb(M^{\leq 2},M_g))
	\]
	induced from the fibre sequence $\Diff_{\partial}(M^{>2}_g)\to\Diff(M_g)\to\Emb(M^{\leq 2},M_g)$ obtained by restriction along the inclusion $\iota_{2}\colon \smash{M^{\le2}}\hookrightarrow M_g$. Arguing as in Assertion \ref{A-emb-homotopy}, using \cref{lem:handle-bounds} \ref{enum:hd-mle2} and \ref{enum:pi1-mg}, and \cref{prop:induction-tower}, the leftmost group in this sequence is polycyclic-by-finite. The same holds for its image in $\pi_0(\Diff_\partial(M_g^{>2}))$ since polycyclic-by-finite groups are closed under taking quotients (see \cref{lem:group-classes-closure}). Moreover, by \cref{lem:path-compo-finite} the set $\pi_0(\Emb(M^{\leq 2},M_g))$ is finite, so we have an extension \[1\lra F \lra \pi_0(\Diff_{\partial}(M^{>2}_g)) \lra G \lra 1\] where $G\le \pi_0(\Diff(M_g))$ is a finite index subgroup and $F$ is polycyclic-by-finite. This implies the claim for $\pi_0(\Diff_{\partial}(M^{>2}_g))$, since (a) being of finite type and having polycyclic solvable subgroups is preserved under extensions and taking finite-index subgroups by Lemmas~\ref{lem:f-infty} and~\ref{lem:polycyclic-sovable-subgroups}, and (b) the groups $F$ and $\pi_0(\Diff(M_g))$ have this property; the former by Lemmas~\ref{lem:polycyclic-sovable-subgroups} and~\ref{lem:finite-type-groups-closure} and the latter by \cref{thm:mcg-finfty} since $M_g$ is a \emph{closed} manifold.
	
	To see the part concerning $\pi_0(\Emb^{\cong}(M^{>2}_g))$, we use the exact sequence
\[
\pi_2(B^2\Diff_\partial(\partial M^{>2}\times I)) \xlra{\partial}\pi_1(B\Diff_\partial(M_g^{>2})) \lra \pi_1(B\Emb^{\cong}(M^{>2}_g))\lra 0.
\]
induced by \eqref{eq:weiss-K_g}, which yields an extension
	\beq
	1\lra \mr{im}(\partial)\lra\pi_0(\Diff_{\partial}(M^{>2}_g))\lra\pi_0(\Emb^{\cong}(M^{>2}_g))\lra 1.
	\eeq
	As $\pi_2(B^2\Diff_\partial(\partial M^{>2}\times I)) $ is abelian, so is $\mr{im}(\partial)$. Since $\pi_0(\Diff_{\partial}(M^{>2}_g))$ has polycyclic solvable subgroups by the first part, the abelian subgroup $\mr{im}(\partial)$ is finitely generated and thus of finite type. Since $\pi_0(\Diff_{\partial}(M^{>2}_g))$ is also of finite type, it follows from \cref{lem:f-infty} that $\pi_0(\Emb^{\cong}(M^{>2}_g))$ is of finite type as well and the proof is finished.\end{proof}
	
Assertion~\ref{A-emb-homology} now follows from two applications of the second part of \cref{prop:finite-type-homotopy-groups}, one to $X=\Emb(M^{\leq 2},M)_{\iota_2}$ and one to $X = B\Emb^{\cong}(M^{\geq 2}_g)$.  In the former case, the hypothesis that $\pi_k(X)$ is of finite type for all $k$ is satisfied as a result of Assertion~\ref{A-emb-homotopy}: for $k\ge2$ since it is abelian and finitely generated, and for $k=1$ since it is polycyclic-by-finite and thus finite type by \cref{lem:finite-type-groups-closure}. In the latter case, the hypothesis the hypothesis that $\pi_k(X)$ is of finite type is satisfied for $k=1$ by \cref{lem:A-mcg} and for $k\ge2$ since it is finitely generated abelian and thus of finite type by Assertion~\ref{A-emb-homotopy}.

\subsection*{Assertion \ref{A-mcg2}: \textnormal{\textit{$\pi_0(\Diff(M_g)_{\iota_2})$ is of finite type}}}
The group $\pi_0(\Diff(M_g)_{\iota_2})$ is the stabiliser of the inclusion $\iota_2\colon M^{\leq 2} \hookrightarrow M_g$ with respect to the action of $\pi_0(\Diff(M_g))$ on $\pi_0(\Emb(M^{\leq 2},M_g))$ by precomposition. As the latter is a finite set by \cref{lem:path-compo-finite}, the subgroup \[\pi_0(\Diff(M_g)_{\iota_2})\le \pi_0(\Diff(M_g))\] has finite index. As $\pi_1(M_g)$ is finite by \cref{lem:handle-bounds} \ref{enum:pi1-mg} and $M_g$ is closed, $\pi_0(\Diff(M_g))$ is of finite type by \cref{thm:mcg-finfty}, so by \cref{lem:f-infty} the same holds for $\pi_0(\Diff(M_g)_{\iota_2})$.

\subsection*{Assertion \ref{A-diff-homology}: \textnormal{\textit{\textnormal{(b)}$_{k}$ implies \textnormal{(c)}$_{k}$}}}
We begin by choosing a Moore--Postnikov $3$-factorisation of a tangent classifier of $M$,
\[
M \xrightarrow{\ell_{M}} B \xlra{\lambda} B\mr{O}(2n),
\]
 i.e.~a factorisation into a $3$-connected cofibration $\ell_{M}$ followed by a $3$-co-connected fibration. As $\pi_3(B\mr{O}(2n))$ vanishes, this factorisation induces a canonical isomorphism
 \begin{equation}\label{equ:homotopy-groups-B}\pi_k(B) \cong \begin{cases} \pi_k(M) & \text{if $k \leq 2$,} \\
 	\pi_k(B\mr{O}(2n)) & \text{if $k \geq 3$.} \end{cases}\end{equation}
Using the notation from \cref{section:GRW}, the map $\ell_{M}$ is covered by an element $\ell_M\in \mr{Bun}(TM,\lambda^*\gamma_{2n})$, which in turn gives an element $\ell_{M_{0,1}}\in \mr{Bun}(TM_{0,1},\lambda^*\gamma_{2n})$ by restriction. Up to homotopy, this extends uniquely to an element $\ell_{M_g}\in \mr{Bun}(TM_g,\lambda^*\gamma_{2n})$ by the following lemma.

\begin{lemma}\label{lem:lambda-str-unique}
The class of $\ell_{M_{0,1}}$ has a unique preimage under the map 
\[\pi_0(\mr{Bun}(TM_{g},\lambda^*\gamma_{2n}))\lra \pi_0(\mr{Bun}(TM_{0,1},\lambda^*\gamma_{2n})) \] induced by the inclusion $M_{0,1}\subset M_g$.
\end{lemma}

\begin{proof}
The element $\ell_{M_{0,1}}$ gives a choice of tangent classifier $\tau_{M_{0,1}}\colon M_{0,1}\ra B\mr{O}(2n)$ which we may extend along the inclusions $M_{0,1}\subset M_{g,1}\subset M_g$ to tangent classifiers $\tau_{M_{g,1}}$ and $\tau_{M_{g}}$ respectively. Identifying the space of bundle maps $\mr{Bun}(TN,\lambda^*\gamma_{2n})$ (up to weak equivalence) with the space of lifts of a tangent classifier $N\ra B\mr{O}(2n)$ along $\lambda\colon B\ra B\mr{O}(2n)$, we see that the claim is equivalent to showing existence and uniqueness of a solution to the lifting problem
\[
\begin{tikzcd}
M_{0,1}\arrow[r,"\ell_{M_{0,1}}"]\arrow[d,hook]&B\arrow[d,"\lambda"]\\[-4pt]
M_{g}\arrow[ur,dashed]\arrow[r,"\tau_{M_{g}}",swap]&B\mr{O}(2n)
\end{tikzcd}
\]
As $\lambda$ is $3$-co-connected and the inclusion $M_{0,1}\subset M_g$ is $(n-1)$-connected, the uniqueness part follows from obstruction theory and the assumption that $n \geq 3$. Existence would also follow from obstruction theory as long as $n > 3$, but to also handle the case $n=3$ we proceed differently. By obstruction theory it suffices to construct a lift on $M_{g,1}\subset M_{g}$. Viewing $M_{g,1}$ as the boundary connected sum $M_{g,1}\cong M_{0,1}\natural W_{g,1}$ with $W_{g,1}\coloneqq (S^n\times S^n)^{\sharp g}\backslash \mr{int}(D^{2n})$, this follows from the fact that $W_{g,1}$ is parallelisable: without loss of generality, $\tau_{M_g}$ and $\ell_{M_{0,1}}$ are constant on $M_{0,1} \cap W_{g,1}$ and $\tau_{W_{g,1}}$ is constant on $W_{g,1}$, so that we may take the lift over $W_{g,1}$ to be constant as well.
\end{proof}
 
We write $B\Diff^\lambda(M_g)_{\ell_{M_g}}\subset B\Diff^\lambda(M_g)$ for the path component induced by the component $[\ell_{M_g}]\in \pi_0(\mr{Bun}(TM_g,\lambda^*\gamma_{2n}))$ ensured by the previous lemma.

\begin{proposition}\label{prop:HDiffK_g fg}
For $k\leq \varphi_{\pi_1(M)}(g)$, the group $H_k(B\Diff^\lambda(M_g)_{\ell_{M_g}};\bZ)$ is finitely generated.
	\end{proposition}
	\begin{proof}
This will follow from \cref{thm:stable-finiteness} once we show that the genus of $(M_g,\ell_{M_g})$ is at least $g$ and that the hypotheses (i)--(iii) of that theorem are satisfied. 

To prove the former, recall that the genus of $(M_g,\ell_{M_g})$ agrees with the maximal number of disjointly embedded copies of $e\colon S^n\times S^n\backslash D^{2n}\hookrightarrow M_g$, such that the map $(\ell_{M_g}\circ e)\colon S^n\times S^n\backslash \mr{int}(D^{2n})\ra B$ is nullhomotopic. Since $W_{g,1} = (S^n\times S^n)^{\sharp g}\backslash \mr{int}(D^{2n})$ has genus $g$ and there is an evident embedding $e':W_{g,1}\hookrightarrow M_g=M\sharp (S^n\times S^n)^{\sharp g}$, it suffices to show that the map $(\ell_{M_g} \circ e')$ is null-homotopic. This follows directly from the proof of \cref{lem:lambda-str-unique}.

To establish the hypotheses (i)--(iii) of \cref{thm:stable-finiteness}, we use \eqref{equ:homotopy-groups-B}. Since $B\mr{O}(2n)$ has f.g.\ homotopy groups and $\pi_k(M)$ is finite for $k\le 1$ and f.g.\ for $k=2$, assumptions (i) and (ii) are clearly satisfied. Postnikov $2$-truncation, the twisted Euler class and the Pontryagin classes further induce a map to a generalised Eilenberg--MacLane-space 
\begin{equation}\label{EM-decomposition}\textstyle{B_\bQ^{\mr{fib}}\lra K\Big(\pi_1(B),\big(\pi_2(B)\otimes\bQ\big)[2]\oplus\bQ[2n]\oplus\bigoplus_{1\le i\le n-1}\bQ[4i]\Big)}\end{equation} where the action of $\pi_1(B)\cong \pi_1(M)$ on $\pi_2(B)\otimes\bQ$ is the usual one and the action on $\pi_*(BO(2n))\otimes\bQ=\bQ[2n]\oplus\bigoplus_{1\le i\le n-1}\bQ[4i]$ is via the map $B\ra BO(2n)$. This map is an equivalence in view of \eqref{equ:homotopy-groups-B}, so assumption (iii) is satisfied and \cref{thm:stable-finiteness} applies.
\end{proof}

The key lemma to reduce Assertion \ref{A-diff-homology} to \cref{prop:HDiffK_g fg} is the following.

\begin{lemma}\label{lem:diff-inc} The image of the canonical map
	\[\pi_1(B\Diff^\lambda(M_g)_{\ell_{M_g}}) \lra \pi_1(B\Diff(M_g))\cong \pi_0(\Diff(M_g))\]
	agrees with the subgroup $\pi_0(\Diff(M_g)_{\iota_2})\subset \pi_0(\Diff(M_g))$.
\end{lemma}

\begin{proof}
	There is a commutative diagram
	\[\begin{tikzcd} \pi_0(\Diff(M_g)) \rar \dar &\pi_0(\mr{Bun}(TM_g,TM_g))\rar\dar & \pi_0(\mr{Bun}(TM_g,\lambda^*\gamma)) \dar{\circled{2}}  \\[-4pt]
		\pi_0(\Emb(M^{\leq 2},M_g)) \arrow{r}{\circled{1}} &\pi_0(\mr{Bun}(TM^{\leq 2},TM_g)) \arrow{r}{\circled{3}} &\pi_0(\mr{Bun}(TM^{\leq 2},\lambda^*\gamma))\end{tikzcd} \]
		where the vertical arrows are induced by restriction, the left horizontal arrows by taking derivatives, and the right horizontal arrows by postcomposition with $\ell_{M_g}$. The image in question agrees with the preimage of $\ell_{M_g}$ under the top horizontal composition and the subgroup $\pi_0(\Diff(M_g)_{\iota_2})$ is the preimage of $\iota_{2}$ under the left vertical map. As the images of $\iota_{2}$ and $\ell_{M_g}$ in the bottom right corner agree by construction, it suffices to show that the preimages under the numbered maps of the respective images of $[\mr{id}]\in \pi_0(\Diff(M_g))$ are singletons. In the proof of \cref{lem:path-compo-finite} we have seen that $\circled{1}$ is a bijection, so there is nothing to show in this case. For  $\circled{2}$, the claim is equivalent to uniqueness up to homotopy of lifts
			\[\begin{tikzcd} M^{\le 2} \dar[hook] \rar{\ell_{M^{\le 2}}} & B \dar \\[-4pt]
		M_g\rar[swap]{\tau_{M_g}}\arrow[ur,dashed] & B\mr{O}(2n).\end{tikzcd}\]
		which follows by obstruction theory since $M_{g}$ can be obtained from $M^{\le 2}$ by attaching handles of index at least $3$ and the right vertical map is $3$-co-connected by construction. Finally, we can show that $\circled{3}$ is injective. This is equivalent to showing the uniqueness up to homotopy of a lift of $\ell_{M^{\le2}}\colon M^{\le2}\ra B$ along $M_g\ra B$, which is another consequence of obstruction theory since $M^{\le2}$ has handle dimension at most $2$ and $M_g\ra B$ is $3$-connected.
\end{proof}

We are now in a position to prove Assertion~\ref{A-diff-homology}. Fixing $k\ge1$ and assuming that $\pi_i(B\Diff(M_g)_{\iota_2})$ is finitely generated for $1\le i\le k$ and $g\ge0$, we need to show that $H_i(B\Diff(M_g)_{\iota_2};\bZ)$ is finitely generated for $1\le i\le k+1$ and $k+1\le \varphi_{\pi_1(M)}(g)$. Using \cref{lem:diff-inc}, we see that there is a fibration sequence
	\[\mr{Bun}(TM_g,\lambda^*\gamma)_{\ell_{M_g}} \lra B\Diff^\lambda(M_g)_{\ell_{M_g}} \lra B\Diff(M_g)_{\iota_2}\]
	where $\mr{Bun}(TM_g,\lambda^*\gamma)_{\ell_{M_g}}\subset \mr{Bun}(TM_g,\lambda^*\gamma)$ is the path component of $\ell_{M_g}$. As explained in the proof of \cref{lem:bundle-maps}, the space $\mr{Bun}(TM_g,\lambda^*\gamma)$ fits into a fibration sequence
	\[
	\mr{Sect}(\mr{Iso}(TM_g,f^*\lambda^*\gamma)\ra M_g)\lra\mr{Bun}(TM_g,\lambda^*\gamma)\lra \mr{Map}(M_g,B)
	\]
	with fibre taken over a map $f$. Using that each path component of $\mr{GL}_{2n}(\bR)$ has finite fundamental group and finitely generated higher homotopy groups, it follows from \cref{lem:section-spaces-finite-CW} that the fibre of this sequence has at all basepoints polycyclic-by-finite fundamental group and finitely generated higher homotopy groups. The same holds for the base and hence also for the total space, since it follows from \eqref{equ:homotopy-groups-B} that also $B$ has finite fundamental group and finitely generated higher homotopy groups. In particular, it follows from \cref{cor:fg-homotopy-implies-fg-homology} that $\mr{Bun}(TM_g,\lambda^*\gamma)_{\ell_{M_g}}$ has degreewise finitely generated integral homology groups. Using that $\pi_i(\Diff(M_g)_{\iota_2})$ is of finite type for $i=0$ by \ref{A-mcg2} and finitely generated for $i\le k$ by assumption, we conclude from \cref{cor:fg-homotopy-implies-fg-homology} that for all $g\ge0$ the $E^2$-page of the Serre spectral sequence
\[
E^2_{p,q}\cong H_p(B\Diff(M_g)_{\iota_2};H_q(\mr{Bun}(TM_g,\lambda^*\gamma)_{\ell_{M_g}};\bZ))\implies H_{p+q}(B\Diff^\lambda(M_g)_{\ell_{M_g}};\bZ).
\]
is finitely generated for $p\le k$ and $q\ge0$. By \cref{prop:HDiffK_g fg}, the group $E^{\infty}_{p,q}$ is finitely generated for $p+q\leq \varphi_{\pi_1(M)}(g)$, so $E^2_{k+1,0}\cong H_{k+1}(B\Diff(M_g)_{\iota_2};\bZ)$ is finitely generated for $k+1\leq \varphi_{\pi_1(M)}(g)$, as claimed.
\medskip

This finishes the proof of Assertion~\ref{A-diff-homology} which was the remaining ingredient in the proof of \cref{thm:pi2 finite}. Equipped with \cref{thm:pi2 finite}, we now proceed to prove \cref{thm:fg pi}.

\subsection{Proof of \cref{thm:fg pi}}\label{sec:proof-of-main-thm}
Let $M$ be a closed smooth manifold of dimension $d=2n\ge6$ with finite fundamental group. Our task is to show that $\BDiff(M)$ and all its homotopy groups are of finite type. We first assume that $M$ is connected.

By \cref{thm:mcg-finfty}, the group $\pi_1(B\Diff(M))$ is of finite type, so in view of \cref{prop:finite-type-homotopy-groups}, it suffices to show that the higher homotopy groups of $\BDiff(M)$ are of finite type or equivalently (as they are abelian) that they are finitely generated. To this end, we consider the composition 
\begin{equation}\label{equ:SW}
\pi_2(M)\xlra{h}H_2(M;\bZ) \xlra{w_2} \bZ/2
\end{equation}
of the Hurewicz map with the homomorphism induced by the second Stiefel-Whitney class of $M$. Since $\pi_2(M)$ is finitely generated (see \cref{coro:pi1-finite-pik-fg}), the kernel of \eqref{equ:SW} is generated by finitely many elements, say $x_1,\ldots,x_k\in \pi_2(M)$. By general position, these generators can be represented (as unpointed homotopy classes) by an embedding $e\colon \sqcup^kS^2\hookrightarrow M$. The normal bundle of this embedding is classified by the second Stiefel-Whitney class, so is trivial as we chose the $x_i$ to lie in the kernel of \eqref{equ:SW}. Consequently, there is an extension $e$ to an embedding of the form $\bar{e}\colon \sqcup ^kS^2\times D^{d-2}\hookrightarrow M$. Restricting diffeomorphisms of $M$ to this tubular neighborhood yields a fibration sequence
\begin{equation}\label{eq:restriction pre-surgery}
\Diff_{\partial}(M\backslash \interior(\sqcup ^kS^2\times D^{d-2}))\lra \Diff(M)\lra \Emb(\sqcup^kS^2\times D^{d-2},M)
\end{equation}
with fibre taken over $\bar{e}$. As $\sqcup^kS^2\times D^{d-2}$ has handle dimension $2\le d-3$ and $M$ has finite fundamental group, at all basepoints the group $\pi_k(\Emb(\sqcup^k S^2\times D^{d-2},M))$ is by \cref{prop:induction-tower} finitely generated for $k\ge2$ and polycyclic-by-finite for $k=1$. In particular, all subgroups of the fundamental groups are polycyclic-by-finite, so in particular finitely generated (see \cref{lem:group-classes-closure}). The long exact sequence in homotopy groups of \eqref{eq:restriction pre-surgery} thus reduces the claim to showing that $\pi_k(\Diff_{\partial}(M\backslash \interior(\sqcup^k S^2\times D^{d-2}));\mr{id})$ is finitely generated for $k\ge1$. For this, we consider \beq
\chi(M) \coloneqq M\setminus\mathrm{int}(\sqcup^kS^2\times D^{d-2})\cup_{\sqcup^kS^2\times S^{d-3}} (\sqcup^k D^3\times S^{d-3}),
\eeq
the manifold obtained from $M$ by performing surgery on the embeddings that constitute $\bar{e}$. This is a closed connected manifold with $\pi_1(\chi(M))\cong \pi_1(M)$ and $\pi_2(\chi(M))\cong\pi_2(M)/\Lambda$ for a subgroup $\Lambda\le \pi_2(M)$ that contains the chosen generating set $\{x_1,\ldots, x_k\}\subset \pi_2(M)$ of the kernel of \eqref{equ:SW} (cf.\,\cite[Lemma 2]{MilnorSurgery}), so $\pi_2(\chi(M))$ is finite, in fact either trivial or of order $2$. Similar to \eqref{eq:restriction pre-surgery}, we have a fibration sequence
\begin{equation}\label{eq:restriction post-surgery}
 \Diff_{\partial}(M\backslash \interior(\sqcup^kS^2\times D^{d-2}))\lra \Diff(\chi(M)) \lra\Emb(\sqcup^kD^3\times S^{d-3},\chi(M)),
\end{equation}
by restricting diffeomorphisms to the canonical inclusion $\sqcup^kD^3\times S^{d-3}\subset \chi(M)$. As $\sqcup^kD^3\times S^{d-3}$ has handle dimension $d-3$ and $\pi_k(\chi(M))$ is finite for $k\le2$, the base of this fibration has at all basepoints finitely generated homotopy groups by \cref{prop:induction-tower}, and the total space has finitely generated homotopy groups based at the identity as a result of \cref{thm:pi2 finite}, so the fibre must have finitely generated homotopy groups based at the identity as well. This completes the proof for connected manifolds.

If $M$ is disconnected, we write it as a disjoint union $M\cong \sqcup_{i\in I}M_i^{\sqcup m_i}$ for a finite set $I$, positive integers $m_i\ge 1$ for $i\in I$, and pairwise non-diffeomorphic connected manifolds $M_i$. This decomposition induces an isomorphism of the form
\[
\textstyle{\Diff(M)\cong \prod_{i\in I}\Diff(M_i)\wr\Sigma_{m_i}},
\]
so the homotopy groups of $\Diff(M)$ fit into short exact sequences 
\[\textstyle{0\lra \prod_{i\in I }\pi_k(\Diff(M_i))\lra \pi_k(\Diff(M))\lra  \prod_{i\in I }\pi_k(\Sigma_{m_i})\lra 0}
\]
for $k\ge0$. By the first part, the kernel of this extension is of finite type and since finite groups are of finite type (see \cref{lem:finite-type-groups-closure}), also the quotient is of finite type. Thus $\pi_k(\Diff(M))$ is of finite type by \cref{lem:f-infty} for $k\ge0$. The proof is completed by applying  \cref{prop:finite-type-fibrations}.

\section{Variants and applications of \cref{thm:fg pi}}\label{sec:elaborations} In this section we prove variants of \cref{thm:fg pi} for manifolds with boundary (see \cref{thm:bdy}) or tangential structures (see \cref{thm:tangential}), and for spaces of homeomorphisms (see \cref{thm:homeo}). We also explain an application to embedding spaces, which includes \cref{thm:codim-one-or-two}.

\subsection{Manifolds with boundary}For manifolds with nonempty boundary, we can prove the following version of \cref{thm:fg pi} for the higher homotopy groups of $\BDiff_\partial(M)$.

\begin{theorem}\label{thm:bdy}Let $d=2n\ge6$ and $M$ a compact smooth $d$-dimensional manifold, possibly with boundary. If at all basepoints $\pi_1(M)$ is finite, then the groups $\pi_k(B\Diff_\partial(M))$ are finitely generated for $k \geq 2$.
\end{theorem}

\begin{proof}We fix a relative handle decomposition of the boundary inclusion $\partial M\subset M$, write $M^{\le 2}\subset M$ for the union of the handles of index $\le 2$, and set $M^{>2}=M\backslash\mr{int}(M^{\le2})$. Restriction to $M^{\leq 2}$ induces a fibre sequence
\[\Diff_\partial(M^{>2}) \lra \Diff_\partial(M) \lra \Emb_\partial(M^{\leq 2},M)\]
with fibre taken over the inclusion $M^{\leq 2} \subset M$. The relative handle dimension of $\partial M\subset M^{\le2}$ at most $2$, and as $2\le d-3$, \cref{prop:induction-tower} implies that all path components of the base of the sequence have finitely generated higher homotopy groups and polycyclic-by-finite fundamental group, so the result follows from the long exact sequence in homotopy groups once we prove that $\pi_k(\Diff_{\partial}(M^{> 2});\mr{id})$ is finitely generated for $k\ge1$. To see this, we consider the double
\[D(M^{> 2}) \coloneqq M^{> 2} \cup_{\partial M^{>2}} M^{> 2},\]
which is a \emph{closed} manifold. Its diffeomorphism group fits into a fibre sequence
\begin{equation}\label{equ:double-fibration}\Diff_\partial(M^{> 2}) \lra \Diff(D(M^{> 2})) \lra \Emb(M^{> 2},D(M^{> 2}))\end{equation}
induced by restriction, with fibre taken over the inclusion $M^{> 2} \subset D(M^{> 2})$.
As the manifold $D(M^{> 2})$ is obtained from $M^{>2}$ by attaching handles of index at least $3$, and $M$ is obtained from $M^{>2}$ by attaching handles of index at least $d-2$, the inclusions $M^{>2}\subset D(M^{>2})$ and $M^{>2}\subset M$ are both $2$-connected. In particular $\pi_1(D(M^{>2}))$ is finite, so $\pi_k(\Diff(D(M^{> 2}));\mr{id})$ is of finite type for $k\ge0$ by \cref{thm:fg pi}, so in particular finitely generated. Now, by \cref{prop:induction-tower}, every path component of the base has finitely generated homotopy groups, since $M^{> 2}$ has handle dimension at most $d-3$, which can be seen by reversing the handle decomposition of $M$. The result follows from the long exact sequence in homotopy groups induced by \eqref{equ:double-fibration}.
\end{proof}

\begin{remark}
By \cref{prop:finite-type-fibrations}, to extend the full statement of \cref{thm:fg pi} to manifolds with boundary, it would suffice to show that $\pi_0(\Diff_\partial(M))$ is of finite type. \cref{thm:mcg-finfty} does not apply if $\partial M\neq \varnothing$ (cf.\,\cref{rem:relative-arithm}), but there is a trick in the case that $\partial M$ admits a nullbordism $W$ of (absolute) handle dimension $\le 1$. In this case, we have an exact sequence
\[
\pi_1(\Emb(W,M\cup_{\partial M}W);\iota)\lra \pi_0(\Diff_\partial(M))\lra \pi_0(\Diff(M\cup_{\partial M} W))_{\iota}\lra 0
\]
where the rightmost group is the stabiliser of the inclusion $\iota\in \pi_0(\Emb(W,M\cup_{\partial M}W))$. By \cref{prop:induction-tower} the leftmost group is polycyclic-by-finite, and so its image in  $\pi_0(\Diff_\partial(M))$ is as well. Hence, by \cref{lem:f-infty} in order to show that $\pi_0(\Diff_\partial(M))$ is of finite type, it suffices to show that $\pi_0{\Diff(M\cup_{\partial M} W)}_{\iota}$ is of finite type, which would follow from \cref{thm:mcg-finfty} and \cref{lem:f-infty} if $\pi_0(\Diff(M\cup_{\partial M} W))_{\iota}\le \pi_0(\Diff(M\cup_{\partial M} W))$ had finite index. This is indeed the case, since $\pi_0(\Emb(W,M\cup_{\partial M}W))$ is finite by an argument similar to the proof of \cref{lem:path-compo-finite}.

This trick applies for instance to manifolds $M^\circ=M\backslash \mr{int}(D^{2n})$ obtained from a closed manifold $M$ of dimension $2n\ge6$ with finite $\pi_1(M)$ by removing an embedded disc.
\end{remark}

\subsection{Tangential structures}Let $M$ be a manifold of dimension $2n$ and 
\[M\xlra{\ell} B\xlra{\lambda} B\mr{O}(2n)\]
be a factorisation of a choice of tangent classifier for $M$. This induces a bundle map $\ell\colon TM\ra \lambda^*\gamma_{2n}$ and, by restriction, a bundle map $\ell_\partial\colon TM|_{\partial M}\ra \lambda^*\gamma_{2n}$. Denoting by $\mr{Bun}_\partial(TM,\lambda^*\gamma;\ell_\partial)$ the space of bundle maps $TM\ra \lambda^*\gamma_{2n}$ extending $\ell_{\partial}$, we consider the homotopy quotient by the action via the derivative
\[B\Diff^\lambda_\partial(M;\ell_\partial) \coloneqq \mr{Bun}_{\partial}(TM,\lambda^*\gamma;\ell_\partial) \sslash \Diff_{\partial}(M)\]
and the path component $B\Diff^\lambda_\partial(M;\ell_\partial)_\ell\subset B\Diff^\lambda_\partial(M;\ell_\partial)$ induced by $\ell$.

\begin{theorem}\label{thm:tangential}Let $M$ be a compact smooth manifold of dimension $2n \geq 6$, possibly with boundary. If at all basepoints $\pi_1(M)$ is finite and at all basepoints $\pi_k(B)$ is finitely generated for $k \geq 2$, then $\pi_k(B\Diff^\lambda_\partial(M;\ell_\partial)_\ell)$ is finitely generated for $k \geq 2$.\end{theorem}

\begin{proof}There is a fibration sequence
	\[\mr{Bun}_\partial(TM,\lambda^*\gamma;\ell_\partial)_\ell \lra B\Diff^\lambda_\partial(M;\ell_\partial)_\ell \lra B\Diff_\partial(M)_{\ell},\]
	where $\Diff_\partial(M)_{\ell}\subset \Diff_\partial(M)$ is the stabiliser of $[\ell]\in \pi_0(\mr{Bun}(TM,\lambda^*\gamma;\ell_\partial))$ by the $\Diff_{\partial}(M)$-action via the derivative, and $\mr{Bun}_\partial(TM,\lambda^*\gamma;\ell_\partial)_\ell \subset \mr{Bun}_\partial(TM,\lambda^*\gamma;\ell_\partial)$ denotes the path component of $\ell$. The base in this fibration has finitely generated higher homotopy groups by \cref{thm:bdy}, so it suffices to show the same for the fibre. Using the fibre sequence
	\[
	\mr{Sect}_\partial(\mr{Iso}(TM,f^*\lambda^*\gamma)\ra M)\ra \mr{Bun}_\partial(TM,\lambda^*\gamma;\ell_\partial)\lra \Map_\partial(M,B)
	\]
	with homotopy fibre taken over a map $f\in \Map_\partial(M,B)$ (see the proof of \cref{lem:bundle-maps}), this follows from \cref{lem:section-spaces-finite-CW} applied to base and fibre.
\end{proof}

\begin{remark}If in addition $\pi_1(B)$ is polycyclic-by-finite and $\pi_0(\Diff_\partial(M)_\ell)$ is of finite type, then a similar argument shows that also the space $B\Diff^\lambda_\partial(M;\ell_\partial)_\ell$ and its fundamental group are of finite type. This applies for instance when $M$ is closed and $\pi_k(\hofib(B\ra B\mr{O}(2n))$ is finite for $0\le k\le 2n$, since then $\pi_0(\mr{Bun}(TM,\lambda^*\gamma;\ell_\partial)_\ell)$ is finite by \cref{lem:section-spaces-finite-CW}, so $\pi_0(\Diff(M)_\ell)$ has finite index in $\pi_0(\Diff(M))$ which has finite type by \cref{thm:mcg-finfty}.
\end{remark}

\subsection{Homeomorphisms} In order to extend our results to topological groups of homeomorphisms $\Homeo_\partial(M)$ in the compact-open topology, we rely on parametrised smoothing theory (see \cite[Essay V]{KirbySiebenmann}) in the form of a fibration sequence
\begin{equation}\label{equ:smoothing-theory}\mr{Sect}_\partial (E \to M)_A \lra B\Diff_\partial(M) \lra B\mr{Homeo}_\partial(M),\end{equation}
for any smooth manifold $M$ of dimension $d\ge5$, where $E \to M$ is a fibration with fibre $\mr{Top}(d)/\mr{O}(d)$ equipped with a section, and $\mr{Sect}_\partial(E \to M)_A \subset \mr{Sect}_\partial(E \to M)$ is a certain collection of path components of the relative section space.

\begin{theorem}\label{thm:homeo}Let $M$ be a compact smooth manifold of dimension $2n \geq 6$ such that at all basepoints $\pi_1(M)$ is finite. Then the groups $\pi_k(B\Homeo_\partial(M))$ are finitely generated for $k\ge2$. Moreover, if $\partial M=\varnothing$ then $B\Homeo(M)$ and $\pi_1(B\Homeo(M))$ are of finite type.
\end{theorem}

\begin{proof}The space $\mr{Top}(2n)/\mr{O}(2n)$ is 1-connected and has finitely generated homotopy groups for $2n\ge6$ as a consequence of \cite[Corollary D]{Kupers}, so \cref{lem:section-spaces-finite-CW} implies that  $\mr{Sect}_\partial (E \to M)_A$ has polycyclic-by-finite fundamental group and finitely generated homotopy groups at all basepoints. As $\pi_k(B\Diff_\partial(M))$ is finitely generated for $k\ge2$ by \cref{thm:bdy}, the long exact sequence induced by \eqref{equ:smoothing-theory} implies the first part of the claim. For the second, it suffices by \cref{prop:finite-type-homotopy-groups} to show that $\pi_0(\Homeo(M))$ is of finite type. This holds by \cref{thm:mcg-finfty}.\end{proof}

\begin{remark}A version for the higher homotopy groups of spaces of homeomorphisms with tangential structure can be deduced from \cref{thm:homeo} similar to the proof of \cref{thm:tangential}.
\end{remark}

\subsection{Embedding spaces in all codimensions}\label{sec:codimension-1-or-2}
\cref{prop:induction-tower}---proved via embedding calculus---is a finiteness result for spaces of embeddings of handle codimension at least $3$
and \cref{thm:fg pi} can be interpreted as a finiteness result for space of embeddings of handle codimension $0$. We now deduce from the latter a finiteness result for any codimension, which includes \cref{thm:codim-one-or-two} as a special case.

We adopt the conventions on embeddings of triads introduced in \cref{section:triads}.

\begin{theorem}\label{thm:thick-codim-one-or-two} Let $M$ be a smooth compact manifold of dimension $2n \geq 6$ and $N\subset M$ a compact triad-pair. If at all basepoints $\pi_1(M)$ and $\pi_1(M\backslash N)$ are finite, then $\pi_k(\mr{Emb}_{\partial_0}(N,M),\mr{inc})$ is finitely generated for $k \geq 2$.
\end{theorem}

\begin{proof}We consider the compact triad-pair $\nu(N)\subset M$ given by a closed tubular neighborhood $\nu(N)\subset M$ where $\partial_0(\nu(N))$ is the part of $\partial(\nu(N))$ lying over $\partial_0N$. Restriction along the inclusion $N\subset \nu(N)$ induces a fibration sequence
\begin{equation}\label{equ:thickening-embeddings}
\mr{Sect}_{\partial_0}(\mr{Iso}(T\nu(N),TM\ \mr{rel}\ TN)\ra N)\lra \Emb_{\partial_0}(\nu(N),M)\lra \Emb_{\partial_0}(N,M)
\end{equation}
with homotopy fibre taken over the inclusion, where $\mr{Sect}_{\partial_0}(\mr{Iso}(T\nu(N),TM\ \mr{rel}\ TN)\ra N)$ is the space of sections fixed at $\partial_0N$ of the bundle $\mr{Iso}(T\nu(N),TM\ \mr{rel}\ TN)\ra N$ over $N$ whose fibres are the space of linear isomorphisms $T_n\nu(N)\to T_nM$ that extend the inclusion $T_nN\subset T_nM$. Each of these fibres are homotopy equivalent to $\mr{GL}_k(\bR)$ where $k=\mr{codim}(N\subset M)$, so it follows from \cref{lem:section-spaces-finite-CW} that the fibre in \eqref{equ:thickening-embeddings} has finitely generated higher homotopy groups and polycyclic-by-finite fundamental group at all basepoints. It thus suffices to show that the higher homotopy groups of $\Emb_{\partial_0}(\nu(N),M)$ based at the inclusion are finitely generated. For this we consider the fibration sequence induced by restriction
	\[\Diff_\partial((M \setminus \nu(N))\cup \partial_1(\nu(N))) \lra \Diff_\partial(M) \lra \Emb_{\partial_0}(\nu(N),M)\]
	with homotopy fibre taken over the inclusion. By \cref{thm:bdy}, the groups $\pi_k(\Diff_\partial(M);\mr{id})$ are finitely generated for $k\ge1$, so it suffices to show that the same holds for $\Diff_\partial((M \setminus \nu(N))\cup \partial_1(\nu(N)))$. Shrinking the tubular neighborhood, we see that $(M \setminus \nu(N))\cup \partial_1(\nu(N))$ is homotopy equivalent to $M\backslash N$, so its fundamental group is finite at all basepoints by assumption. The claim now follows from \cref{thm:bdy}.
\end{proof} 

\subsubsection{Some examples}The hypotheses for \cref{thm:codim-one-or-two} are met in many examples in which \cref{prop:induction-tower} does not apply. In the introduction we have an example involving hypersurfaces in $\bC P^n$ with possibly non-simply connected complement; here are two even more basic ones:

\begin{enumerate}
\item \label{exam:cpn-emb} The complement of the usual inclusion $\bC P^{n-1} \subset \bC P^n$ is contractible, so \cref{thm:codim-one-or-two} implies that for $n\ge3$, the groups $\pi_k(\Emb(\bC P^{n-1},\bC P^n),\mr{inc})$ are finitely generated for $k \geq 2$. Furthermore, since the complement of a closure of a tubular neighborhood of $\bC P^{n-1}$ is diffeomorphic to a $2n$-disc and $\pi_0(\Diff_\partial(D^{2n}))$ is finite, it also follows that $\pi_1(\Emb(\bC P^{n-1},\bC P^n),\mr{inc})$ is polycyclic-by-finite.
\item Similarly, we get that $\pi_k(\Emb(\bR P^{2n-1},\bR P^{2n}),\mr{inc})$ is finitely generated if $k \geq 2$ and polycyclic-by-finite if $k=1$, provided $2n\ge6$. 
\end{enumerate}
In \cref{exam:codim-two-infinite} below, we will see that \cref{thm:thick-codim-one-or-two} can fail in low codimensions if the assumption on the fundamental group is dropped.

\section{On the limits of Theorems~\ref{thm:fg pi} and \ref{thm:codim-one-or-two}}\label{sec:infinite-generatedness}
In this final section, we explain some examples that illustrate that the finiteness condition on the fundamental groups in \cref{thm:fg pi} or \cref{thm:codim-one-or-two} is often necessary.

\subsection{Infinite generation in the homotopy groups of diffeomorphism groups}
As mentioned in the introduction, it is known that the homotopy groups of $\BDiff_\partial(M)$ need not be finitely generated if the fundamental group is not finite. Perhaps the most prominent instances of this phenomenon are high-dimensional tori and solid tori:

\begin{example}\label{ex:torus}\ 
\begin{enumerate}
\item\label{enum:torus} For $M = T^d=\times ^dS^1$ with $d \geq 6$,  Hsiang--Sharpe \cite[Theorem 2.5]{HsiangSharpe} proved that for $k\leq d$, the group $\pi_k(B\Diff(T^d))$ is not finitely generated (see also \cite[Theorem 4.1]{Hatcher} for the case $k=1$).  Moreover, they showed that the groups $\pi_k(B\Diff(T^{d-2} \times S^2))$ are not finitely generated for $k \leq d-2$, again provided that $d \geq 6$ \cite[Example 3]{HsiangSharpe}.
\item\label{enum:solid-torus} One can show that for every prime $p$ the sum $\bigoplus_\bN \bZ/p$ injects into $\pi_{2p-3}(B\Diff_\partial(S^1 \times D^{d-1}))$ as long as $d$ is sufficiently large with respect to $p$. For $p=2$, this can be deduced from work of Igusa \cite[8.a.2]{Igusa} and for $p>2$ it can be shown by a combination of pseudoisotopy theory and work of Grunewald--Klein--Macko \cite{GrunewaldKleinMacko} (see \cite[Section 4]{FarrellOntaneda} for an explanation of how this goes, and see \cite[Theorem A]{BustamanteRandalWilliams} for a different approach if $d$ is even).
\end{enumerate}\end{example}

In addition to specific examples such as tori, there are large classes of manifolds with infinite fundamental group for which some homotopy group of $B\Diff(M)$ is not finitely generated:

\begin{example}\ 
\begin{enumerate}
\item Hsiang--Sharpe \cite[Proposition 2.2 (A)]{HsiangSharpe} proved that for any closed manifold $M$ of dimension $d \geq 6$ such that $\pi_1(M)$ contains infinitely many elements that are not conjugate to their inverses, the group $\pi_0(\Diff(M))$ is not finitely generated if $\pi_1(\hAut(M),\mr{id})$ is finitely generated (so for example if at all basepoints $\pi_k(M)$ is polycyclic-by-finite for $k=1$ and finitely generated for $2\le k\le d+1$, by \cref{lem:section-spaces-finite-CW}).
\item Once completed, Weiss--Williams' programme on automorphisms of manifolds (see \cite{WeissWilliamsSurvey} for a survey) can be used to generalise \cref{ex:torus} \ref{enum:solid-torus} to show non-finite generation results for $B\Diff_\partial(M)$ for many manifolds $M$, e.g.~all high-dimensional orientable manifolds whose fundamental group surjects onto the integers. For this kind of manifolds one can show that for any prime $p>2$, the group $\pi_{2p-3}(B\Diff_\partial(M))$ is not finitely generated if the dimension $d$ of $M$ is sufficiently large compared to $p$. 

Let us outline how this roughly goes: using the composition 
\[\tfrac{\BlockDiff_\partial(S^1\times D^{d-1})}{\Diff_\partial(S^1\times D^{d-1})} \xrightarrow{\mathrm{ww}} \Omega^{\infty}(\Omega \Wh^\Diff(S^1)_{hC_2}) \xrightarrow{\mathrm{tr}} \Omega^{\infty}(\Omega \Wh^\Diff(S^1))\]
of the Weiss--Williams map and the transfer, one sees that the group $\pi_{2p-3}(\BlockDiff_\partial(S^1\times D^{d-1})/\Diff_\partial(S^1\times D^{d-1}))$ injects into $\pi_{2p-3}(\Omega^{\infty}(\Omega \Wh^\Diff(S^1)_{hC_2}))$ for $p\ll d$, and after localising away from $2$ also into $\pi_{2p-2}(\Wh^\Diff(S^1))$, which is known to contain a copy of $\bigoplus_\bN \bZ/p$ (related to \cref{ex:torus} \ref{enum:solid-torus}). Choosing an embedding of $S^1\times D^{d-1}\hookrightarrow M$ such that $S^1\times D^{d-1}\hookrightarrow M \ra K(\pi_1(M),1)\ra K(\bZ,1)$ is an equivalence, one then argues that the same is true for $\pi_{2p-3}(\BlockDiff_\partial(M)/\Diff_\partial(M))$, using the commutative diagram
\[\qquad \begin{tikzcd} \tfrac{\BlockDiff_\partial(S^1\times D^{d-1})}{\Diff_\partial(S^1\times D^{d-1})} \rar{\mathrm{tr} \circ \mathrm{ww}} \dar &[5pt] \Omega^{\infty}(\Omega \Wh^\Diff(S^1)) \dar \rar{\simeq} &[-10pt] \Omega^{\infty}(\Omega \Wh^\Diff(K(\bZ,1)))\\[-5pt]
	\tfrac{\BlockDiff_\partial(M)}{\Diff_\partial(M)} \rar{\mathrm{tr} \circ \mathrm{ww}} & \Omega^{\infty}(\Omega \Wh^\Diff(M)) \rar & \Omega^{\infty}(\Omega \Wh^\Diff(K(\pi_1(M),1))) \uar \end{tikzcd}\]
whose bottom-right horizontal map is induced by Postnikov truncation. To go from $\BlockDiff_\partial(M)/\Diff_\partial(M)$ to $B\Diff_\partial(M)$ it suffices to show that the composition $\BlockDiff_\partial(M)\ra \BlockDiff_\partial(M)/\Diff_\partial(M) \ra \Omega^{\infty}(\Omega \Wh^\Diff(K(\bZ,1)))$ is trivial on higher homotopy groups, which one can do by factoring it over $\hAut(K(\bZ,1))\simeq K(\bZ,1)\rtimes \bZ^\times$.

\end{enumerate}
\end{example}

\subsection{Infinite generation in the homotopy groups of embedding spaces}

We end with an example illustrating that the assumption on $\pi_1(M \setminus N)$ \cref{thm:thick-codim-one-or-two} is necessary.

\begin{example}\label{exam:codim-two-infinite}For many $d\ge6$, the component of standard inclusion in $\Emb_{\partial}(D^{d-2},D^d)$ has some homotopy groups that are not finitely generated. In this case the complement is homotopy equivalent to $S^1$, so its fundamental group is infinite cyclic. 

To see this non-finite generation, one can argue as follows: restricting diffeomorphisms of $S^1\times D^{d-1}$ to a slice $\ast\times D^{d-1}\subset S^1\times D^{d-1}$ yields a fibre sequence 
\[\Diff_{\partial}(D^d)\to\Diff_{\partial}(S^1\times D^{d-1})\to\Emb_{\partial}(D^{d-1},S^1\times D^{d-1})\]
with fibre taken over the inclusion. Now fix a pair $(k,d)$ such that $d\neq 4,5,7$ and so that $\pi_k(\Diff_\partial(S^1\times D^{d-1});\mr{id})$ is not finitely generated (as mentioned in \cref{ex:torus} \ref{enum:solid-torus}, there are many such pairs). Since $\pi_k(\Diff_\partial(D^{d});\mr{id})$ is finitely generated for all $k\ge0$ by \cite[Theorem A]{Kupers}, the group $\pi_k(\Emb_{\partial}(D^{d-1},S^1\times D^{d-1});\mr{inc})$ cannot be finitely generated. From this and a variant of the ``delooping trick'' (cf.\,\cite[p.\,23-25]{BurgheleaLashofRothenberg}) in the form of an equivalence $\Emb_{\partial}(D^{d-1},S^1\times D^{d-1})\simeq \Omega \Emb_{\partial}(D^{d-2},D^{d})$, one concludes $\pi_{k+1}(\Emb_{\partial}(D^{d-2},D^{d});\mr{inc})$ is not finitely generated either.
\end{example}

\bibliographystyle{amsalpha}
\bibliography{refs}

%\vspace{+0.2cm}
\end{document}